\documentclass[11pt,a4paper, oneside,reqno]{amsproc}
\usepackage{graphicx}
\usepackage{float}
\usepackage{pstricks}
\usepackage{amscd}
\usepackage{amsmath}
\usepackage{amsxtra}
\usepackage{hyperref}
\usepackage[T1]{fontenc}
\usepackage[utf8]{inputenc}
\usepackage{lipsum}
\usepackage{tikz}
\usepackage{amssymb, amsthm}
\usepackage{url}
\usepackage{enumerate}
\usepackage{mathrsfs}
\usepackage{epsfig}
\usepackage[active]{srcltx}

\usetikzlibrary{arrows}
\usetikzlibrary{calc}
\newcommand{\doi}[1]{doi:\,\href{http://dx.doi.org/#1}{#1}}

\theoremstyle{plain}
\newtheorem{theorem}{Theorem}[section]
\newtheorem{c-theorem}{Construction theorem}[section]

\newtheorem{lemma}[theorem]{Lemma}
\newtheorem{proposition}[theorem]{Proposition}
\newtheorem{corollary}[theorem]{Corollary}

\theoremstyle{definition}

\newtheorem{example}[theorem]{Example}
\newtheorem{question}[theorem]{Question}

\theoremstyle{remark}
\newtheorem{remark}[theorem]{Remark}
\newlength{\struh}
\newlength{\textminustop}
\newcommand{\mf}{\mathfrak}

\newcommand*{\child}[1]{\mathsf{Chi}(#1)}
\newcommand*{\childn}[2]{{\mathsf{Chi}}^{\langle#1\rangle}(#2)}

\newcommand*{\parentn}[2]{{\mathsf{par}}^{\langle#1\rangle}(#2)}

\newcommand*{\gammab}{\boldsymbol\gamma}
\newcommand*{\Ge}{\geqslant}
\newcommand*{\lambdab}{\boldsymbol\lambda}

\newcommand*{\Le}{\leqslant}
\newcommand*{\parent}[1]{\mathsf{par}(#1)}

\makeatletter
\@namedef{subjclassname@2020}{%
\textup{2020} Mathematics Subject Classification}
\makeatother

\newcommand{\ncom}{\newcommand}
\ncom{\bq}{\begin{equation}}
\ncom{\eq}{\end{equation}}
\ncom{\beqn}{\begin{eqnarray*}}
\ncom{\eeqn}{\end{eqnarray*}}
\ncom{\beq}{\begin{eqnarray}}
\ncom{\eeq}{\end{eqnarray}}
\ncom{\nno}{\nonumber}
\ncom{\rar}{\rightarrow}
\ncom{\Rar}{\Rightarrow}
\ncom{\noin}{\noindent}
\ncom{\bc}{\begin{centre}}
\ncom{\ec}{\end{centre}}
\ncom{\sz}{\scriptsize}
\ncom{\rf}{\ref}
\ncom{\sgm}{\sigma}
\ncom{\Sgm}{\Sigma}
\ncom{\dt}{\delta}
\ncom{\Dt}{Delta}
\ncom{\lmd}{\lambda}
\ncom{\Lmd}{\Lambda}
\ncom{\eps}{\epsilon}
\ncom{\pcc}{\stackrel{P}{>}}
\ncom{\dist}{{\rm\,dist}}
\ncom{\im}{{\rm Im\,}}
\ncom{\sgn}{{\rm sgn\,}}
\ncom{\ba}{\begin{array}}
\ncom{\ea}{\end{array}}
\ncom{\eop}{\hfill{{\rule{2.5mm}{2.5mm}}}}
\ncom{\eof}{\hfill{{\rule{1.5mm}{1.5mm}}}}
\ncom{\hone}{\mbox{\hspace{1em}}}
\ncom{\htwo}{\mbox{\hspace{2em}}}
\ncom{\hthree}{\mbox{\hspace{3em}}}
\ncom{\hfour}{\mbox{\hspace{4em}}}
\ncom{\hsev}{\mbox{\hspace{7em}}}
\ncom{\vone}{\vskip 2ex}
\ncom{\vtwo}{\vskip 4ex}
\ncom{\vonee}{\vskip 1.5ex}
\ncom{\vthree}{\vskip 6ex}
\ncom{\vfour}{\vspace*{8ex}}
\ncom{\norm}{\|\;\;\|}
\ncom{\integ}[4]{\int_{#1}^{#2}\,{#3}\,d{#4}}
\ncom{\inp}[2]{\langle{#1},\,{#2} \rangle}
\ncom{\Inp}[2]{\Langle{#1},\,{#2} \Langle}
\ncom{\vspan}[1]{{{\rm\,span}\#1 \}}}
\ncom{\dm}[1]{\displaystyle {#1}}

\begin{document}

\title[Wold-type decomposition for left-invertible weighted shifts]{Wold-type decomposition for left-invertible weighted shifts on a rootless directed tree}

\author[S. Chavan]{Sameer Chavan}
\author[S. Trivedi]{Shailesh Trivedi}

\address{Department of Mathematics and Statistics\\
Indian Institute of Technology Kanpur, 208016, India}

\email{chavan@iitk.ac.in}
  
\address{Department of Mathematics, Birla Institute of Technology and Science, Pilani 333031, India}

 \email{shailesh.trivedi@pilani.bits-pilani.ac.in}

\dedicatory{Dedicated to the memory of Pradeep Kumar}

\thanks{The work of the second author is supported partially through the SERB-SRG  (SRG/2023/000641-G) and OPERA (FR/SCM/03-Nov-2022/MATH)}



\begin{abstract} 
Let $S_{\lambdab}$ be a bounded left-invertible weighted shift on a rootless directed tree $\mathcal T=(V, \mathcal E).$ We address the question of when $S_{\lambdab}$ has Wold-type decomposition. We relate this problem to the convergence of the series $\displaystyle {\tiny \sum_{n = 1}^{\infty} \sum_{u
\in G_{v, n}\backslash G_{v, n-1}} \Big(\frac{\lambdab^{(n)}(u)}{\lambdab^{(n)}(v)}\Big)^2},$ $v \in V,$ involving the  moments $\lambdab^{(n)}$ of $S^*_{\lambdab}$, where $G_{v, n}=\childn{n}{\parentn{n}{v}}.$ The main result of this paper characterizes all bounded left-invertible weighted shifts $S_{\lambdab}$ on $\mathcal T,$ which have Wold-type decomposition. 
%As an application, we construct a family of analytic norm-increasing $3$-expansions without Wold-type decomposition. 
\end{abstract}

\maketitle

%\tableofcontents

\section{Introduction and preliminaries}
Let $\mathbb N$, $\mathbb Z,$ and $\mathbb C$ stand for
the sets of non-negative integers, integers, and complex
numbers, respectively. 
%The complex conjugate of a complex number $w$ will be denoted by $\overline{w}.$ 
For a subset $A$ of a non-empty set $X$, $\mbox{card}(A)$ denotes the cardinality of $A$. For a subset $E$ of a vector space $X,$ let $\mbox{span}\,E$ denote the linear span of $E.$
For a complex Hilbert space $\mathcal H,$ let $\mathcal B(\mathcal H)$ denote the unital $C^*$-algebra of bounded linear operators on $\mathcal H.$ Let $\lambda \in \mathbb C$ and $T \in \mathcal B(\mathcal H).$ A sequence $\{u_n\}_{n \Ge 1}$ in $\mathcal H$ is called a {\it Weyl sequence} for $(\lambda, T)$ if $\{u_n\}_{n \Ge 1}$ is a weakly null sequence of unit vectors such that $(T-\lambda)u_n \rar 0$ as $n \rar \infty.$  
For a subspace $\mathcal M$ of $\mathcal H,$ let $[\mathcal M]_T=\vee \{T^nh : n \Ge 0, ~h \in \mathcal M \},$ where $\vee$ stands for the closed linear span.
The {\it hyper-range} of $T$ is given by $T^{\infty}(\mathcal H) = \bigcap_{n=0}^{\infty}T^n\mathcal H.$
We say that $T$ is {\it analytic} if $T^{\infty}(\mathcal H) = \{0\}$.
An operator $T$ has the {\it wandering subspace property} if $\mathcal H = [\ker T^*]_T,$ where $\ker S$ stands for the kernel of $S \in \mathcal B(\mathcal H).$  
Since $\ker T^*$ is orthogonal to $T^n(\ker T^*)$ for any integer $n \Ge 1,$ 
following \cite{H1961}, $\ker T^*$ is called the {\it wandering subspace} of $T.$
Assume that $T \in \mathcal B(\mathcal H)$ is a left-invertible operator. 
Following \cite{S2001}, we say that $T$ has {\it Wold-type decomposition} if $T^{\infty}(\mathcal H)$ reduces $T$ to a unitary and $\mathcal H = T^{\infty}(\mathcal H) \oplus [\ker T^*]_T.$  In this case, the restriction of $T$ to $[\ker T^*]_T$ is referred to as the {\it analytic part} of $T.$ The {\it Cauchy dual} $T'$ of $T$ is given by  $T':=T(T^*T)^{-1}$ (this notion appeared implicitly in \cite[Proof of Theorem~1]{R1988}, and formally introduced in \cite[Eq.~(1.6)]{S2001}).  

Given a positive integer $m$ and $T \in \mathcal B(\mathcal H)$, set
   \beqn
\Delta_{m, T} := \sum_{k=0}^m (-1)^k \binom{m}{k}{T^*}^kT^k.
   \eeqn
An operator $T$ is said to be an {\it isometry} (resp. {\it norm-increasing}) if $\Delta_{1, T} = 0$ (resp. $\Delta_{1, T} \Le 0$). 
For $m \Ge 2$, an operator
$T \in \mathcal B(\mathcal H)$ is said to be an {\it $m$-isometry} (resp. {\it $m$-concave}, resp. {\it $m$-expansion}) if $\Delta_{m, T} = 0$ (resp. $(-1)^m \Delta_{m, T} \leqslant 0$, resp. $\Delta_{m, T} \leqslant 0$). 
%A $2$-concave operator is commonly referred to as a {\it concave} operator.
%It is easy to see that any $m$-concave operator is left-invertible; hence, its Cauchy dual is well-defined. 
The reader is referred to \cite{AS1995, R1988, S2001} for the basic properties of $m$-isometries and related classes.

In \cite[p.~185]{S2001}, Shimorin asked whether it is true that the $m$-concave and
norm-increasing operators have Wold-type decomposition for $m \Ge 3.$
Here we explain the thought process that could help in answering this question. 
In this regard, note that the construction in \cite[Example~3.1]{ACT2020} was capitalized on the idea of the weighted
shift on a one-circuit directed graph (see \cite[Section 3]{BJJS2017}). 
The one-dimensional space on which the weighted shift in question fails to be norm-increasing corresponds to the indicator function of an element of the circuit. Thus, to get a counterexample to Shimorin's question from weighted shifts on directed graphs, it is necessary to confine ourselves to directed graphs without circuits. 
This thinking can indeed be strongly based on \cite[Theorem~4.3]{ACT2020}, which provide a necessary condition for the failure of the wandering subspace property for analytic norm-increasing $m$-concave operators. Indeed, we have a more general fact:

\begin{theorem} 
\label{dichotomy-new-coro}
Let $T \in \mathcal B(\mathcal H)$ be an analytic, norm-increasing $m$-concave operator and let $\mathcal M = \mathcal H \ominus [\ker T^*]_T.$   
If $T$ does not have the wandering subspace property, then for any $\lambda \in \sigma_{ap}(T'|_{\mathcal M}),$ there exists a Weyl sequence in $\mathcal M$ for both $(\lambda, T)$ and $(\lambda, T').$ In this case, $\mathcal M$ is infinite-dimensional.
\end{theorem}

Theorem~\ref{dichotomy-new-coro} is not directly related to the objective of this paper, and hence we include its proof in the appendix. In view of Theorem~\ref{dichotomy-new-coro}, a counterexample to Shimorin's question can not be built on any one-circuit directed graph (see \cite[Corollary~4.4]{ACT2020}). This, combined with the fact that any bounded weighted shift on a rooted directed tree admits Wold-type decomposition (see \cite[Proposition~1.3.4]{CPT2017}),
propels us to explore rootless directed graphs without a circuit or more specifically, rootless directed trees. The natural question arises as to which choice of the rootless directed tree would lead to the desired counterexample. 
It turns out that a rootless directed tree of finite branching index (see \cite[Definition~3]{CT2016}) does not support analytic weighted shifts. Indeed, by \cite[Lemma~6.1]{CT2016}, the indicator function of the so-called generalized root belongs to the hyper-range of the weighted shift in question. This motivates us to examine rootless directed trees of infinite branching index. A basic example of such directed trees, with which we are primarily concerned, is that of the rootless 
quasi-Brownian directed tree of finite valency (see \cite[Example 4.4]{J2003}). Although we do not have definite answer to Shimorin's question, the main result of this paper (see Theorem~\ref{Wold-type-thm}) answers the following closely related question$:$
\begin{question}
Which left-invertible weighted shifts on a rootless directed tree have Wold-type decomposition ?
\end{question}

Before we state the main result of this paper, we reproduce some necessary notions from the graph theory. The reader is referred to \cite{JJS2012} for definitions and basic properties pertaining to directed trees. The child function and parent partial function are denoted by $\child{\cdot}$ and $\parent{\cdot},$ respectively. If the directed tree in question is rootless, then $\parent{\cdot}$ is a function.

Throughout this paper, $\mathcal T =(V, \mathcal E)$ denotes a countably infinite, leafless, rootless directed tree. A {\it path} $\mathcal P$ in $\mathcal T$ is a subtree $(V_{\mathcal P}, \mathcal E_{\mathcal P})$ of $\mathcal T$ such that for every $v \in V_{\mathcal P},$ $\mbox{card}(\mathsf{Chi}_{\mathcal P}(v))=1.$ 
By a {\it bilateral path in $\mathcal T$}, we understand the path of the form $\{v_m\}_{m \in \mathbb Z},$ where $v_0 \in V$ and $v_m,$ $m \neq 0,$ is given by\footnote{The definition of $v_m,$ $m \Ge 1,$ uses the axiom of choice together with the assumption that $\mathcal T$ is leafless.} 
\beq \label{v-m-def}
v_m := \begin{cases} 
%v & \mbox{if~}m=0, \\
\parentn{-m}{v_0} & \mbox{if~} m \Le -1, \\
\mbox{any but fixed element of} ~\child{v_{m-1}} & \mbox{if~}m \Ge 1.
\end{cases}
\eeq
We say that  $u,v \in V$ {\it belong to the same generation} of $\mathcal T$ if there exists a non-negative integer $n$ such that $\parentn{n}{u} = \parentn{n}{v}.$ This defines an equivalence relation on $V$. The {\it generation} $\mathscr G_v$ of $v \in V$ is defined as 
the equivalence class containing $v.$ 
For $v \in V$ and an integer $n \in \mathbb N$, define 
\beq \label{A-v-n-def} A(v,n) := \begin{cases}
\{v\} & \mbox{ if } n = 0,\\
\childn{n}{\parentn{n}{v}} \setminus \childn{n-1}{\parentn{n-1}{v}} & \mbox{ if } n \geqslant 1.\end{cases}
\eeq 
For a net $\{\lambda_v\}_{v\in V},$ define $\lambdab^{(n)}$ by  
\beq \label{lambda-k}
\lambdab^{(0)} := 1, \quad \lambdab^{(n)} := \lambdab (\lambdab \circ \mathsf{par}) \cdots (\lambdab \circ \mathsf{par}^{\langle n-1\rangle}), \quad n \Ge 1,
\eeq
where $\lambdab : V \rar \mathbb C$ is given by $\lambdab(v)=\lambda_v,$ $v \in V.$ 
For every $v \in V,$ define $\alpha_{\lambdab}(v) \in (0, \infty]$ by
\beq \label{alpha-m-lambdab}
\alpha_{\lambdab}(v) := \displaystyle \sum_{n = 0}^{\infty} \sum_{u
\in A(v,n)} \left(\frac{\lambdab^{(n)}(u)}{\lambdab^{(n)}(v)}\right)^2.
\eeq
%\beq \label{alpha-m-lambdab}
%\alpha_{\lambdab}(v) = \displaystyle \sum_{n = 0}^{\infty} \frac{1}{\lambdab^{(n)}(v)} \sum_{u
%\in A(v,n)} \left(\lambdab^{(n)}(u)\right)^2.
%\eeq
%For $m \in \mathbb Z,$ consider the extended real number 
%\beq \label{alpha-m-lambdab}
%\alpha_m(\lambdab, \bf v) = \displaystyle\sum_{n = 0}^{\infty} \sum_{u \in A(v_m,n)} \left(\frac{\lambdab^{(n)}(u)}{\lambdab^{(n)}(v_m)}\right)^2.
%%\quad \alpha'_m = \displaystyle \sum_{n = 0}^{\infty} \sum_{u \in A(v_m,n)} \left(\frac{\lambdab'^{(n)}(u)}{\lambdab'^{(n)}(v_m)}\right)^2.
%\eeq
%\begin{remark}
%It turns out that for any $v, w$ belonging to same generations, $\alpha_{\lambdab}(v) < \infty$ if and only if $\alpha_{\lambdab}(w) < \infty.$ 
%This is the reason that we do not show dependence of a path in the notation $\alpha_{\lambdab}(\cdot).$ 
%\end{remark}
%where $A(v_m,n),$ $v_m$ and $\lambdab^{(n)}$ are given by \eqref{A-v-n-def}-\eqref{lambda-k}, respectively.

%{\color{orange} Questions. Whether $\alpha_m$ is independent of the choice of $A(v_m,n),$ $v_m$ ? Can one see directly if $\alpha_m(\lambdab) < \infty$ if and only if $\alpha_m(\lambdab') < \infty$ ?}

%Let $\mathcal T =(V, \mathcal E)$ is a countably infinite leafless, rootless directed tree. 
In what follows, $\ell^2(V)$ stands for the Hilbert
space of square summable complex functions on $V$
equipped with the standard inner product. If $e_u$
denotes the indicator function of $\{u\},$ then 
$\{e_u\}_{u\in V}$ is an
orthonormal basis of $\ell^2(V).$ Given a system
$\lambdab = \{\lambda_v\}_{v\in V}$ of complex numbers, 
we define the {\em weighted shift operator} $S_{\lambdab}$ on ${\mathcal T}$
with weights $\lambdab$ by
   \begin{align*}
   \begin{aligned}
{\mathscr D}(S_{\lambdab}) & := \{f \in \ell^2(V) \colon
\varLambda_{\mathcal T} f \in \ell^2(V)\},
   \\
S_{\lambdab} f & := \varLambda_{\mathcal T} f, \quad f \in {\mathscr
D}(S_{\lambda}),
   \end{aligned}
   \end{align*}
where $
(\varLambda_{\mathcal T} f) (v) :=
\lambda_v \cdot f\big(\parent v\big),$ $v \in V.$ This notion has been extensively studied in \cite{JJS2012}. 
We always assume that $\{\lambda_v\}_{v\in V}$ consists of positive numbers and $S_{\lambdab}$ belongs to $\mathcal B(\ell^2(V)).$
Recall from \cite{ACJS2019, BDPP2019, CPT2017} that the weighted shift $S_{\lambdab}$ on a rootless directed tree is {\it balanced} if $\|S_{\lambdab} e_v\| = \|S_{\lambdab} e_u\|$ for all $u$ and $v$ belonging to the same generation. 
If $S_{\lambdab}$ is a left-invertible weighted shift with weights 
$\{\lambda_v\}_{v\in V},$ then by the discussion following \cite[Lemma~1.1]{CT2016}, the Cauchy dual $S'_{\lambdab}$ of 
$S_{\lambdab}$ is a bounded weighted shift on $\mathcal T$ with weights 
$\lambdab'=\{\lambda'_v\}_{v\in V},$ where 
\beq \label{weights}
\lambda'_v:=\frac{\lambda_v}{\|S_\lambda 
e_{\mathsf{par}(v)}\|^2}, \quad v \in V.
\eeq 
The following is the main result of this paper.
\begin{theorem}\label{Wold-type-thm}
Let $S_{\lambdab}$ be a left-invertible weighted shift on a rootless directed tree $\mathcal T=(V, \mathcal E).$ Let $\alpha_{\lambdab}$ and $\lambdab'$ be as given in \eqref{alpha-m-lambdab} and \eqref{weights}, respectively.
Then $S_{\lambdab}$ has Wold-type decomposition if and only if exactly one of the following holds$:$
\begin{itemize}
\item[(i)] $\alpha_{\lambdab}(v)=\alpha_{\lambdab'}(v)=\infty$ for some $($or equivalently, for all$)$ $v \in V,$ 
\item[(ii)] $\alpha_{\lambdab}(v) < \infty$ for some $($or equivalently, for all$)$ $v \in V,$ $\lambda_{v}=$ ${\tiny \sqrt{\frac{\alpha_{\lambdab}(\parent{v})}{\alpha_{\lambdab}(v)}}}$ for all $v$ in a bilateral path $($or equivalently, for all $v \in V)$,  and $S_{\lambdab}$ is a balanced weighted shift on $\mathcal T.$
\end{itemize}
\end{theorem}  
\begin{remark}
In either of the cases (i) or (ii), the analytic part of $S_{\lambdab}$ is unitarily equivalent to the operator $\mathscr M_z$ of multiplication by the coordinate function $z$ on a reproducing kernel Hilbert space of $\ker S^*_{\lambdab}$-valued holomorphic functions (see \cite[pp.~153-154]{S2001}). Moreover, if (ii) holds, then it can be seen that the analytic part of $S_{\lambdab}$ is unitarily equivalent to an operator-valued weighted shifts with invertible operator weights (cf. \cite[Proof of Theorem~4.5]{ACJS2019-2}; see Theorem~\ref{balanced-reducing}). 
\end{remark}

Our proof of Theorem~\ref{Wold-type-thm} relies on several results presented in Sections~\ref{S2}-\ref{S4}. Specifically, these include the following two basic facts$:$ 
\begin{itemize}
\item Description of the hyperange of a bounded left-invertible weighted shift on a rootless directed tree (see Theorem~\ref{hyper-range-tree}), 
\item Wold-type decomposition theorem for a bounded left-invertible balanced weighted shift on a rootless directed tree (see Theorem~\ref{balanced-reducing}). 
\end{itemize}
In Section~\ref{S5}, we capitalize on Theorem~\ref{Wold-type-thm} to construct a family of analytic norm-increasing $3$-expansions without Wold-type decomposition (see Example~\ref{Exam-ce}). We conclude the paper with an appendix that includes a general fact (see Theorem~\ref{dichotomy-new}) generalizing Theorem~\ref{dichotomy-new-coro}.

\section{Graph theoretic necessities} \label{S2}

This section collects some results needed in the proof of Theorem~\ref{Wold-type-thm}. 
We begin with several elementary graph-theoretic facts. 
\begin{lemma} \label{preparatory}
Let $\mathcal T =(V, \mathcal E)$ be a rootless directed tree and let $v \in V.$ For $n \in \mathbb N,$ let $A(v, n)$ be given by \eqref{A-v-n-def}.
%For an integer $n \geqslant 0$, define 
%$$A(v,n) := \begin{cases}
%\{v\} & \mbox{ if } n = 0,\\
%\childn{n}{\parentn{n}{v}} \setminus \childn{n-1}{\parentn{n-1}{v}} & \mbox{ if } n \geqslant 1.\end{cases}$$ 
%Further, consider the bi-infinite sequence of vertices $\{v_m\}_{m \in \mathbb Z}$ described as follows:
%$v_0 := v$, $v_{-m} := \parentn{m}{v}$ for $m \geqslant 1$ and $v_m$ is any but fixed element of $\child{v_{m-1}}$ if $m \geqslant 1$.
Then the following statements are true:
\begin{itemize}
\item[(i)] $\mathscr G_v = \bigsqcup_{n\in \mathbb N} A(v,n)$, 
\item[(ii)] $V = \bigsqcup_{m \in \mathbb Z} \mathscr G_{v_m}$ for any bilateral path $\{v_m\}_{m \in \mathbb Z},$
%where\footnote{The definition of $v_m,$ $m \Ge 1,$ uses the axiom of choice together with the assumption that $\mathcal T$ is leafless.} 
%\beq \label{v-m-def}
%v_m = \begin{cases} v & \mbox{if~}m=0, \\
%\parentn{-m}{v} & \mbox{if~} m \Le -1, \\
%\mbox{any but fixed element of} ~\child{v_{m-1}} & \mbox{if~}m \Ge 1,
%\end{cases}
%\eeq
\item[(iii)] $\child{A(\parent{v},n)} = A(v,n+1)$ for all integers $n \geqslant 1$  and $v \in V$,
\item[(iv)] $\parent{A(v,n)} = A(\parent{v},n-1)$ for all integers $n \geqslant 1$  and $v \in V$,
\item[(v)] if $w \in A(v, n)$ for some $n \in \mathbb N,$ then $A(w, j) \subseteq A(v, n)$ for all integers $j = 0, \ldots, n.$
\end{itemize}
\end{lemma}
\begin{proof}
(i) It follows from \cite[Proposition~2.1.12(v)]{JJS2012} that
\beqn
%\label{generation}
\mathscr G_v =\bigcup_{n =0}^{\infty} \childn{n}{\parentn{n}{v}}, \quad v \in V. 
\eeqn Since $\childn{k}{\parentn{k}{v}} \subseteq \childn{l}{\parentn{l}{v}}$ for $k \leqslant l,$ (i) follows from \eqref{A-v-n-def}.

(ii) This may be easily deduced from \cite[Proposition~2.1.12(vi)$\&$(vii)]{JJS2012} (by letting $u = v_0$). 

(iii) $\&$ (iv) These are routine verification using  
\beqn
\childn{n}{v} = \childn{n-1}{\child v}, \, \parentn{n}{v} = \parentn{n-1}{\parent v}, \, v \in V, \, n \Ge 1.
\eeqn

(v) If $n=0,$ then $v=w,$ and we are done. Suppose that $n \Ge 1.$ Let $w \in A(v, n)$ for some $n \in \mathbb N$ and fix $j \in \{0, \ldots, n\}$. Then 
$$A(w, j) \subseteq \childn{j}{\parentn{j}{w}} \subseteq \childn{j}{\parentn{j}{A(v, n)}} = A(v, n),$$
where the last equality follows from repeated applications of (iii) and (iv).  
\end{proof}
\begin{remark} \label{remark-alpha-finite}
Suppose that the cardinality of a generation $\mathscr G_v$ is finite for some $v \in V$. Then by Lemma~\ref{preparatory}(i), there exists $n_0 \in \mathbb N$ such that $A(v,n) = \emptyset$ for all $n \geqslant n_0$. Also, since $A(v,n) \subseteq \mathscr G_v$ for all $n \in \mathbb N$, it follows that $A(v, n)$ is finite for all $n \in \mathbb N$. Hence, $\alpha_{\lambdab}(v) < \infty$ (see \eqref{alpha-m-lambdab}). 
%It now follows from Theorem ?? that $\alpha_{\lambdab}(v) < \infty$ for all $v \in V$. 
\end{remark}
\begin{example}
Let $k$ be any positive integer. Consider the rootless directed tree $\mathcal T_{k,\infty}= (V, \mathcal E)$ described as follows (the reader is referred to \cite[Chapter 6]{JJS2012} for more details about such trees):
\beqn
V &=& \{(0,0), (m,n) : m \geqslant 1, \, 1 \Le n \Le k\} \cup \{(-m, 0) : m \geqslant 1\}; \\
\mathcal E &=& \{\big((0,0), (1,j)\big) : 1\leqslant j \leqslant k\} \cup \{\big((m,n), (m+1,n)\big) : m, n \geqslant 1\}\\
&& \cup\, \{\big((-m,0), (-m+1, 0)\big) : m \geqslant 1\}.
\eeqn
An example of $\mathcal T_{3,\infty}$ is depicted in the following figure.
\begin{figure}[H]
\includegraphics[scale=.7]{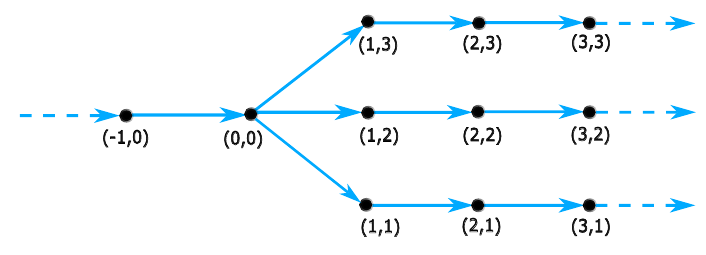}\caption{The 
rootless directed tree $\mathcal T_{3,\infty}$}
\end{figure}
\noindent
Note that the cardinality of $\mathscr G_v$ is finite for all $v \in V$. Hence, by Remark~\ref{remark-alpha-finite}, $\alpha_{\lambdab}(v) < \infty$ for all $v \in V$. 
%Thus, $S_{\lambdab}$ is always non-analytic on $\mathscr T_{k,\infty}$. 
\eof
\end{example}

The following proposition says that if $\alpha_{\lambdab}$ is finite at some vertex $v$ then it is finite at every vertex in $\mathcal G_v.$
\begin{proposition} \label{generation-invariance}
For $v, w$ belonging to the same generation, 
$\alpha_{\lambdab}(v) < \infty$ if and only if $\alpha_{\lambdab}(w) < \infty.$ 
\end{proposition}
\begin{proof} Assume that $v, w \in V$ belong to the same generation. 
%Then $v$ and $w$ belong to two paths.   
%As in the proof of \cite[Proposition 5.11]{ACJS2019}, we consider the path $\{v_m\}_{m \in \mathbb Z}$ described as follows:
%\beq \label{v-m-def}
%v_m = \begin{cases} v & \mbox{if~}m=0, \\
%\parentn{-m}{v} & \mbox{if~} m \Le -1, \\
%\mbox{any but fixed element of} ~\child{v_{m-1}} & \mbox{if~}m \Ge 1,
%\end{cases}
%\eeq
%where, in the last part, we used the axiom of choice together with the assumption that $\mathcal T$ is leafless. 
%%The result of the preceding theorem does not depend on the choice of the sequence $\{v_m\}_{m \in \mathbb Z}$. 
%To see the conclusion, suppose that $\{w_m\}_{m \in \mathbb Z}$ is another sequence constructed in the similar fashion of \eqref{v-m-def}, where $w_0=w.$ 
%We may assume that $w_m$ and $v_m$ lie in the same generation for all $m \in \mathbb Z$ (if not, then we can relabel $\{w_m\}_{m \in \mathbb Z}$). Now, let $m$ be an arbitrary but fixed integer. 
Thus, there exists $n \in \mathbb N$ 
such that $\parentn{n}{v} = \parentn{n}{w}$. Let $n_0$ be the least integer such that 
\beq\label{vm-wm-1}
\parentn{n_0}{v} = \parentn{n_0}{w}.
\eeq 
 This gives that
$\childn{n}{\parentn{n}{v}} = \childn{n}{\parentn{n}{w}}$ for all $n \geqslant n_0,$ 
and hence, by \eqref{A-v-n-def}, 
\beq\label{A-vm-wm}
A(v, n) = A(w, n) \ \mbox{ for all }\ n \geqslant n_0+1.
\eeq
Since $\mathscr G_{v} = \bigsqcup_{n\in \mathbb N} A(v, n)$ (see Lemma~\ref{preparatory}(i)), $w \in A(v, k)$ for some $k \in \mathbb N$. If $k=0,$ then $v=w,$ and we are done. Suppose that $k \Ge 1.$ Hence, by \eqref{A-v-n-def}, we get that
\[\parentn{k}{w} = \parentn{k}{v} \ \mbox{ and }\ \parentn{k-1}{w} \ne \parentn{k-1}{v}.\]
Since $n_0$ is the least integer such that $\parentn{n_0}{w} = \parentn{n_0}{v}$, it follows that $k = n_0$. Hence, $w \in A(v, n_0)$. 
%Now, let
%\[g_m = \sum_{n \in \mathbb N} \sum_{u \in A(v_m,n)} \frac{\lambdab^{(n)}(u)}{\lambdab^{(n)}(v_m)} e_u \ \mbox{ and }\ \tilde g_m = \sum_{n \in \mathbb N} \sum_{u \in A(w_m,n)} \frac{\lambdab^{(n)}(u)}{\lambdab^{(n)}(w_m)} e_u.\]
%We claim that $g_m$ is square summable if and only if $\tilde g_m$ is square summable. To see this, 
Observe that
\beqn
\alpha_{\lambdab}(v) &=& 
%\sum_{n \in \mathbb N} \sum_{u \in A(v_m,n)} \Big(\frac{\lambdab^{(n)}(u)}{\lambdab^{(n)}(v_m)}\Big)^2 \\ 
%=& 
\sum_{n=0}^{n_0} \Big(\sum_{u \in A(v, n)} \frac{\lambdab^{(n)}(u)}{\lambdab^{(n)}(v)}\Big)^2 + \sum_{n = n_0+1}^{\infty} \Big(\sum_{u \in A(v, n)} \frac{\lambdab^{(n)}(u)}{\lambdab^{(n)}(v)}\Big)^2\\
&\overset{\eqref{vm-wm-1}\&\eqref{A-vm-wm}} =& \sum_{n=0}^{n_0} \sum_{u \in A(v, n)} \Big(\frac{\lambdab^{(n)}(u)}{\lambdab^{(n)}(v)}\Big)^2 \\
&+& \Big(\frac{\lambdab^{(n_0)}(w)}{\lambdab^{(n_0)}(v)}\Big)^2 \sum_{n = n_0+1}^{\infty} \sum_{u \in A(w, n)} \Big(\frac{\lambdab^{(n)}(u)}{\lambdab^{(n)}(w)}\Big)^2.
\eeqn
Assume now that $\alpha_{\lambdab}(v) < \infty.$ It then follows from the identity above that 
\beq \label{tail} \displaystyle \sum_{n = n_0+1}^{\infty} \sum_{u \in A(w, n)} \Big(\frac{\lambdab^{(n)}(u)}{\lambdab^{(n)}(w)}\Big)^2 < \infty.
\eeq
%Thus, to establish that $\alpha_{\lambdab}(w_m) < \infty$, it remains to show that $$\displaystyle \sum_{n=0}^{n_0} \sum_{u \in A(w_m,n)} \Big(\frac{\lambdab^{(n)}(u)}{\lambdab^{(n)}(w_m)}\Big)^2 < \infty.$$ 
Further, note that for any $j \in \{0, 1, \ldots, n_0\}$, 
\beqn
&&\sum_{u \in A(w, j)} \Big(\frac{\lambdab^{(j)}(u)}{\lambdab^{(j)}(w)}\Big)^2 \\
&=& \Big(\frac{\lambdab^{(n_0)}(v)}{\lambdab^{(j)}(w) \lambdab^{(n_0-j)}(\parentn{j}{w})}\Big)^2 \sum_{u \in A(w, j)} \Big(\frac{\lambdab^{(n_0)}(u)}{\lambdab^{(n_0)}(v)}\Big)^2, 
\eeqn
where in the above equality, we used the fact that if $u \in A(w, j)$, then $\parentn{j}{u} = \parentn{j}{w}$. Further, by Lemma~\ref{preparatory}(v), $A(w, j) \subseteq A(v, n_0)$ for all $0 \Le j \Le n_0$. Thus, for $0 \Le j \Le n_0,$ the above equality gives that 
\beqn
&&\sum_{u \in A(w, j)} \Big(\frac{\lambdab^{(j)}(u)}{\lambdab^{(j)}(w)}\Big)^2 \\
& \leqslant & \Big(\frac{\lambdab^{(n_0)}(v)}{\lambdab^{(j)}(w) \lambdab^{(n_0-j)}(\parentn{j}{w})}\Big)^2 \sum_{u \in A(v, n_0)} \Big(\frac{\lambdab^{(n_0)}(u)}{\lambdab^{(n_0)}(v)}\Big)^2.
\eeqn
This proves that $\displaystyle\sum_{n=0}^{n_0} \sum_{u \in A(w, n)} \Big(\frac{\lambdab^{(n)}(u)}{\lambdab^{(n)}(w)}\Big)^2 < \infty,$  and hence, by \eqref{tail}, we get $\alpha_{\lambdab}(w) < \infty.$ Interchanging the roles of $v$ and $w$, we obtain that if $\alpha_{\lambdab}(w) < \infty,$ then $\alpha_{\lambdab}(v) < \infty.$ This completes the proof. 
\end{proof}

\section{Structure of hyper-range} \label{S3}

We present below the main result of this section, which is the first step towards our proof of Theorem~\ref{Wold-type-thm}.
\begin{theorem}\label{hyper-range-tree}
Let $S_{\lambdab}$ be a left-invertible weighted shift on a rootless directed tree $\mathcal T=(V, \mathcal E).$ For a bilateral path ${\mf v}=\{v_m\}_{m \in \mathbb Z}$ in $\mathcal T$ $($see \eqref{v-m-def}$),$ let $g_{m, \mf v},$ $m \in \mathbb Z$, be given by 
\beq \label{gm-def}
g_{m, \mf v} = \begin{cases} \displaystyle \sum_{n = 0}^{\infty} \sum_{u \in A(v_m,n)} \frac{\lambdab^{(n)}(u)}{\lambdab^{(n)}(v_m)} e_u 
& \mathrm{if}~ \alpha_{\lambdab}(v_m)< \infty, \\
0 & \mathrm{otherwise}.
\end{cases}
\eeq 
Then $g_{m, \mf v} \in \ell^2(V)$ for every $m \in \mathbb Z$ and  
\beq \label{structure-hr}
 S_{\lambdab}^\infty(\ell^2(V)) &=& \displaystyle  \bigoplus_{m \in \mathbb Z} \mathrm{span}\,\{g_{m, \mf v}\}, \\ \label{action-hr}
S_{\lambdab}\, g_{m, \mf v} &=& \lambda_{v_{m+1}} g_{m+1, \mf v}, \quad m \in \mathbb Z.
\eeq
%Moreover, if $S_{\lambdab}$ is non-analytic, then the following are equivalent$:$
%\begin{itemize}
%\item[(i)] $S_{\lambdab}^\infty(\ell^2(V))$ reduces $S_{\lambdab},$ 
%\item[(ii)] $S_{\lambdab}$ is balanced,
%\item[(iii)] $S_{\lambdab}$ satisfies the kernel condition.
%\end{itemize}
%Furthermore, if $S_{\lambdab}$ is analytic, then 
%the following are equivalent$:$
%\begin{itemize}
%\item[(iv)] $S_{\lambdab}$ has the wandering subspace property, 
%\item[(v)] $\displaystyle \sum_{n = 0}^{\infty} \sum_{u \in A(v_m,n)} \left(\frac{\lambdab'^{(n)}(u)}{\lambdab'^{(n)}(v_m)}\right)^2$ is divergent for every $m \in \mathbb Z,$
%\item[(vi)] $\displaystyle\sum_{n = 0}^{\infty} \sum_{u \in A(v_m,n)} \left(\frac{\lambdab'^{(n)}(u)}{\lambdab'^{(n)}(v_m)}\right)^2$ is divergent for some $m \in \mathbb Z.$
%\end{itemize}
\end{theorem}  

To prove Theorem~\ref{hyper-range-tree}, we need a variant of \cite[Theorem 2.1.1]{K2011} and \cite[Lemma 2.3]{ACT2020} (see \cite[Remark~45]{BJJS1}). First, let $\lambdab^{(n)}$ be as in \eqref{lambda-k}, and 
recall from \cite[Lemma~5.4]{ACJS2019}, \cite[Lemma~6.1.1]{JJS2012} that 
for any positive integer $n,$ 
\beq 
\label{known-rmk-new-new-new}
S_{\lambdab}^{*m}S_{\lambdab}^m e_v = \|S^n_{\lambdab}e_{v}\|^2 e_v,  \quad v \in V, \\
\label{known-rmk-new}
S^n_{\lambdab}e_{u} = \displaystyle \sum_{v \in \childn{n}{u}} \lambdab^{(n)}(v) e_v, \quad u \in V, \\ \label{known-rmk-new-new}
\|S^n_{\lambdab}e_{u}\|^2 = \displaystyle \sum_{v \in \childn{n}{u}} \big(\lambdab^{(n)}(v)\big)^2, \quad u \in V.
\eeq 
We now give a simple description of the hyper-range of a left-invertible weighted shift on a rootless directed tree (cf. \cite[Theorem 2.1.1]{K2011}, \cite[Remark 45]{BJJS1} and \cite[Lemma 2.3]{ACT2020}).
\begin{lemma}\label{hyper-range}
Let $S_{\lambdab}$ be a left-invertible weighted shift on a rootless directed tree $\mathcal T=(V, \mathcal E)$. Then, for any integer $n \Ge 1,$ 
\beq \label{n-range-eq} 
S_{\lambdab}^n(\ell^2(V)) = \left\{ f \in \ell^2(V) : \frac{f}{\lambdab^{(n)}}\Big|_{\childn{n}{v}}  \mbox{ is constant for all } v \in V\right\},
\eeq
where $\lambdab^{(n)}$ is given by \eqref{lambda-k}.
%\beq \label{lambda-k}
%\lambdab^{(n)} := \lambdab (\lambdab \circ \mathsf{par}) \cdots (\lambdab \circ \mathsf{par}^{\langle n-1\rangle}), \quad n \Ge 1.
%\eeq
Moreover, 
\beq\label{hyper-range-eq} 
S_{\lambdab}^\infty(\ell^2(V)) &=& \bigg\{ f \in \ell^2(V) : \frac{f}{\lambdab^{(n)}}\Big|_{\childn{n}{v}}  \mbox{ is constant for each } v \in V \notag\\
&& \mbox{ and for every integer } n \geqslant 1\bigg\}.
\eeq
\end{lemma}
\begin{proof}
Let $n$ be a positive integer. Since $\parent{\cdot}$ is a function, the inclusion $\subseteq$ in \eqref{n-range-eq} can be obtained by imitating the proof of \cite[Lemma 2.3]{ACT2020}. For the remaining inclusion, 
let $f \in \ell^2(V)$ be such that $\frac{f}{\lambdab^{(n)}}\Big|_{\childn{n}{v}}$ is constant for each  $v \in V$. Define $g : V \to \mathbb C$ by
\beqn
g(v) = \frac{f(u)}{\lambdab^{(n)}(u)}, \quad v \in V~\mbox{and}~u \in \childn{n}{v}.
\eeqn
Since $V = \bigsqcup_{v \in V} \childn{n}{v}$, $g$ is well-defined and 
\beqn
\|f\|^2 &=& \sum_{v \in V} \sum_{u \in \childn{n}{v}} |f(u)|^2\\
&=& \sum_{v \in V} |g(v)|^2 \sum_{u \in \childn{n}{v}} \big(\lambdab^{(n)}(u)\big)^2 \\
& \overset{\eqref{known-rmk-new-new}}= &  \sum_{v \in V} |g(v)|^2 \|S_{\lambdab}^n e_v\|^2 \\  
& \geqslant & \Big(\inf_{v \in V} \|S^n_{\lambdab} e_v\|\Big)^2 \sum_{v \in V} |g(v)|^2.
\eeqn
This combined with the left-invertibility of $S^n_{\lambdab}$ (since so is $S_{\lambdab}$) implies that $g \in \ell^2(V)$ and $S^n_{\lambdab}(g)=f.$ Finally, \eqref{hyper-range-eq} follows from \eqref{n-range-eq}. 
\end{proof}

\begin{proof}[Proof of Theorem~\ref{hyper-range-tree}]
Fix $m \in \mathbb Z.$ By parts (i) and (ii) of Lemma~\ref{preparatory}, $g_{m, \mf v} \in \ell^2(V)$ (see \eqref{gm-def}). We claim that $g_{m, \mf v}$ belongs to the hyper-range of $ S_{\lambdab}.$  
In view of Lemma~\ref{hyper-range}, it suffices to verify that for every integer $n \Ge 1,$ 
\begin{align} \label{3.1-new}
& \mbox{$\displaystyle \frac{g_{m, \mf v}}{\lambdab^{(n)}}\Big|_{\childn{n}{v}}$ is constant for each $v \in V$}.
\end{align}
By Lemma~\ref{preparatory}(i),  
\beq \label{support-disjoint}
\mbox{the support of $g_{m, \mf v}$ is either empty or it is equal to $\mathscr G_{v_m}$.}
\eeq 
Hence we may assume that $g_{m, \mf v} \neq 0$ and only consider those $v \in V$ and integers $n \geqslant 1$ for which $\childn{n}{v} \cap \mathscr G_{v_m} \neq \emptyset.$ 
We contend that for some integer $k,$
\beq
\label{child-A-vm}
\childn{n}{v} \subseteq A(v_m, n+k).
\eeq
By Lemma~\ref{preparatory}(ii), $v \in \mathscr G_{v_l}$ for some $l \in \mathbb Z.$ Thus $\childn{n}{v} \subseteq \mathscr G_{v_{l+n}},$ and hence once again by Lemma~\ref{preparatory}(ii), $\mathscr G_{v_{l+n}} = \mathscr G_{v_m}.$ Equivalently, $l + n = m$, and hence, $v \in \mathscr G_{v_{m-n}}$. By Lemma~\ref{preparatory}(i), $v \in A(v_{m-n}, k)$ for some $k \in \mathbb N$. 
This together with Lemma~\ref{preparatory}(iii) now yields \eqref{child-A-vm}. 
Now, let $w \in \childn{n}{v}$ and note that by \eqref{gm-def} and \eqref{child-A-vm}, $g_{m, \mf v} (w)=\frac{\lambdab^{(n+k)}(w)}{\lambdab^{(n+k)}(v_m)}.$ It follows that 
\beqn
\frac{g_{m, \mf v}(w)}{\lambdab^{(n)}(w)} 
%&  =& \frac{1}{\lambdab^{(n)}(w)}\frac{\lambdab^{(n+k)}(w)}{\lambdab^{(n+k)}(v_m) } \\
\overset{\eqref{lambda-k}} =
%\frac{\frac{\lambdab (\lambdab \circ \mathsf{par}) \cdots (\lambdab \circ \mathsf{par}^{\langle n+k-1\rangle}(w))}{\lambdab^{(n+k)}(v_m) }}{\lambdab (\lambdab \circ \mathsf{par}) \cdots (\lambdab \circ \mathsf{par}^{\langle n-1\rangle})(w)}
\frac{(\lambdab \circ \mathsf{par}^{\langle n \rangle}(w)) \cdots (\lambdab \circ \mathsf{par}^{\langle n+k-1\rangle}(w))}{\lambdab^{(n+k)}(v_m)} 
%=\frac{\prod_{j=0}^{k-1} \lambda_{\parentn{n+j}{w}}}{\lambdab^{(n+k)}(v_m)} 
= \frac{\lambdab^{(k)}(v)}{\lambdab^{(n+k)}(v_m)},
\eeqn
%\lambdab (\lambdab \circ \mathsf{par}) \cdots (\lambdab \circ \mathsf{par}^{\langle n-1\rangle})
which is constant. This completes the verification of \eqref{3.1-new}. In turn, this shows that   
$g_{m, \mf v} \in S_{\lambdab}^\infty(\ell^2(V)),$ and hence  
\beq \label{first-inclusion}
\bigvee \{g_{m, \mf v} : m \in \mathbb Z\} \subseteq S_{\lambdab}^\infty(\ell^2(V)).
\eeq 

To see the reverse inclusion in \eqref{first-inclusion}, let $f \in S_{\lambdab}^\infty(\ell^2(V)).$ By parts (i) and (ii) of Lemma~\ref{preparatory},  
\beq\label{f-eq}
f = \sum_{m = -\infty}^{\infty} \sum_{n = 0}^{\infty}\, \sum_{u \in A(v_m,n)}  f(u) e_u.
\eeq
Since 
$v_m \in \childn{n}{\parentn{n}{v_m}},$ $n \geqslant 1,$ by Lemma~\ref{hyper-range}, 
\beqn
\frac{f(u)}{\lambdab^{(n)}(u)} = \frac{f(v_m)}{\lambdab^{(n)}(v_m)}, \quad u \in A(v_m,n).
\eeqn
Substituting this into \eqref{f-eq}, we obtain 
\beqn
f = \sum_{m = -\infty}^{\infty} f(v_m) \Big(\sum_{n = 0}^{\infty} \, \sum_{u \in A(v_m,n)}  \frac{\lambdab^{(n)}(u)}{\lambdab^{(n)}(v_m)} e_u\Big) = \sum_{m \in \mathbb Z} f(v_m) g_{m, \mf v},
\eeqn
and hence, $f \in \bigvee\{g_{m, \mf v} : m \in \mathbb Z\}.$ This together with \eqref{first-inclusion}, now shows that $S_{\lambdab}^\infty(\ell^2(V))=\bigvee\{g_{m, \mf v} : m \in \mathbb Z\}.$ Also, by \eqref{support-disjoint} and Lemma~\ref{preparatory}(ii), $\{g_{m, \mf v}\}_{m \in \mathbb Z}$ is orthogonal. This completes the proof of the first part. 

To see the remaining part, we may assume that $S_{\lambdab}$ is not analytic. Thus there exists $m \in \mathbb Z$ such that $g_{m, \mf v} \neq 0.$ By parts (iii) and (iv) of  Lemma~\ref{preparatory},
\beqn
S_{\lambdab}g_{m, \mf v} &\overset{\eqref{known-rmk-new}}=& \sum_{n = 0}^{\infty} \, \sum_{u \in A(v_m,n)} \frac{\lambdab^{(n)}(u)}{\lambdab^{(n)}(v_m)} \sum_{w \in \child u} \lambda_w e_w \\
&=& \sum_{w \in \child{v_m}} \lambda_w e_w + \sum_{n = 1}^{\infty} \, \sum_{u \in A(v_m,n)} \frac{\lambdab^{(n)}(u)}{\lambdab^{(n)}(v_m)} \sum_{w \in \child u} \lambda_w e_w \\
&=& \lambda_{v_{m+1}} e_{v_{m+1}} +  \sum_{w \in A(v_{m+1}, 1)} \lambda_w e_w \\
&+& \sum_{n = 1}^{\infty} \, \sum_{w \in A(v_{m+1}, n+1)} \frac{\lambda_w\lambdab^{(n)}(\parent w)}{\lambdab^{(n)}(v_m)} e_w \\
&\overset{\eqref{lambda-k}}=& \lambda_{v_{m+1}}\sum_{n = 0}^{\infty} \, \sum_{u \in A(v_{m+1},n)} \frac{\lambdab^{(n)}(u)}{\lambdab^{(n)}(v_{m+1})} e_u \\
 &=& \lambda_{v_{m+1}}g_{m+1, \mf v}.
\eeqn
This gives \eqref{action-hr} provided $g_m \neq 0.$ 
Since $S_{\lambdab}$ is injective, this also shows that $g_{k, \mf v} \neq 0$ for every integer $k \Ge m.$ 
If there exists a smallest integer $m_0$ such that $g_{k, \mf v} \neq 0$ for every $k \Ge m_0,$ then by \eqref{structure-hr}, $g_{m_0, \mf v}\neq S_{\lambdab}(f)$ for any $f \in S_{\lambdab}^\infty(\ell^2(V)).$ 
However, $S_{\lambdab}({S_{\lambdab}^\infty(\ell^2(V))})=S_{\lambdab}^\infty(\ell^2(V)),$ which is not possible. 
Thus $g_{m, \mf v} \neq 0$ for every $m \in \mathbb Z,$ completing the proof of \eqref{action-hr}. 
\end{proof}
\begin{remark} \label{new-rem} Please note the following$:$
\begin{enumerate}
\item The last paragraph of the proof of Theorem~\ref{hyper-range-tree} shows that if $g_{m_0, \mf v} \neq 0$ for some $m_0 \in \mathbb Z,$ then $g_{m, \mf v} \neq 0$ for every $m \in \mathbb Z.$ This combined with \eqref{gm-def} yields that if $\alpha_{\lambdab}(v_{m_0}) < \infty$ for some $m_0 \in \mathbb Z,$ then $\alpha_{\lambdab}(v_{m}) < \infty$ for every $m \in \mathbb Z.$
\item Let ${\mf v}=\{v_m\}_{m \in \mathbb Z}$ and ${\mf w}=\{w_m\}_{m \in \mathbb Z}$ be two bilateral paths in $\mathcal T.$ By Theorem~\ref{hyper-range-tree},
\beqn 
S_{\lambdab}^\infty(\ell^2(V)) = \bigoplus_{m \in \mathbb Z} \mathrm{span}\,\{g_{m, \mf v}\} = \displaystyle  \bigoplus_{m \in \mathbb Z} \mathrm{span}\,\{g_{m, \mf w}\}, \\ 
S_{\lambdab}\, g_{m, \mf v} = \lambda_{v_{m+1}} g_{m+1, \mf v}, ~ S_{\lambdab}\, g_{m, \mf w} = \lambda_{w_{m+1}} g_{m+1, \mf w}, \quad m \in \mathbb Z.
\eeqn
%If $U_{\mf v, \mf w} : S_{\lambdab}^\infty(\ell^2(V)) \rar S_{\lambdab}^\infty(\ell^2(V))$ is the unitary mapping $e_{m, \mf v}:=\frac{g_{m, \mf v}}{\|g_{m, \mf v}\|}$ to $e_{m, \mf w}:=\frac{g_{m, \mf w}}{\|g_{m, \mf w}\|},$ $m \in \mathbb Z,$ then for every $m \in \mathbb Z,$ 
%\beqn
%&&U_{\mf v, \mf w}S_{\lambdab}e_{m, \mf v}=\frac{\|g_{m+1, \mf v}\|}{\|g_{m, \mf v}\|}\lambda_{v_{m+1}} e_{m+1, \mf w}, \\
%&& S_{\lambdab}U_{\mf v, \mf w}e_{m, \mf v}= \frac{\|g_{m+1, \mf w}\|}{\|g_{m, \mf w}\|}\lambda_{w_{m+1}} e_{m+1, \mf w},
%\eeqn
\end{enumerate}
\end{remark}

As a corollary to Theorem~\ref{hyper-range-tree}, we obtain the following dichotomy$:$
\begin{corollary} \label{dicho-coro}
Let $S_{\lambdab}$ be a left-invertible weighted shift on a rootless directed tree $\mathcal T=(V, \mathcal E)$. Then, 
exactly one of the following holds$:$
\begin{itemize}
\item[(i)] $\alpha_{\lambdab}(v)= \infty$ for all $v \in V$ 
$($equivalently, $S_{\lambdab}$ is analytic$),$ 
\item[(ii)] $\alpha_{\lambdab}(v) < \infty$ for all $v \in V$ $($equivalently, $S_{\lambdab}^\infty(\ell^2(V))$ is of infinite dimension$).$  
\end{itemize}
\end{corollary}
\begin{proof} We claim that 
\beq \label{one-all}
\mbox{$\alpha_{\lambdab}(v) < \infty$ for some $v \in V$, then $\alpha_{\lambdab}(w) < \infty$ for every $w \in V.$}
\eeq
Suppose that $\alpha_{\lambdab}(v) < \infty$ for some $v \in V$. Let $\{v_m\}_{m \in \mathbb Z}$ be a bilateral path with $v_0=v.$ By Remark~\ref{new-rem}, $\alpha_{\lambdab}(v_m) < \infty$ for every $m \in \mathbb Z.$ It follows from Proposition~\ref{generation-invariance} that for every $m \in \mathbb Z,$ 
$\alpha_{\lambdab}(w) < \infty$ for every $w \in \mathcal G_{v_m}.$ An application of Lemma~\ref{preparatory}(ii) now shows that $\alpha_{\lambdab}(w) < \infty$ for every $w \in V.$ Thus the claim stands verified. 

By Theorem~\ref{hyper-range-tree}, \eqref{gm-def} and \eqref{one-all}, 
\begin{itemize}
\item $\alpha_{\lambdab}(v)= \infty$ for all $v \in V$ if and only if  $S_{\lambdab}$ is analytic. 
\item $\alpha_{\lambdab}(v) < \infty$ for all $v \in V$ if and only if $S_{\lambdab}^\infty(\ell^2(V))$ is of infinite dimension. 
\end{itemize}
Finally, note that if (i) does not hold, then 
by \eqref{one-all}, (ii) holds. 
%To complete the proof, we must show that $g_m \neq 0$ for every $m \in \mathbb Z.$ Suppose, on the
%contrary, that $g_{m_0} = 0$ for some $m_0 \in \mathbb Z.$   Since $S_{\lambdab}g_m = \lambda_{v_{m+1}}g_{m+1}$ and $S_{\lambdab}|_{S_{\lambdab}^\infty(\ell^2(V))}$ is invertible, $g_{m}=0$ for every $m \in \mathbb Z,$ which is not possible. 
\end{proof}
%\begin{remark} \label{remark-rootless}
%If $S_{\lambdab}$ is not analytic, then $g_m \neq 0$ or equivalently, $\alpha_{\lambdab}(v_m) < \infty$ for every $m \in \mathbb Z.$ 
%\end{remark}

\section{Proof of Theorem~\ref{Wold-type-thm}} \label{S4}
The next step in the proof of Theorem~\ref{Wold-type-thm} provides a Wold-type decomposition theorem for left-invertible, non-analytic, balanced weighted shifts on a rootless directed tree (cf. \cite[Theorem~1.1]{SF}).  
\begin{theorem}[Wold-type decomposition] \label{balanced-reducing}
Let $S_{\lambdab}$ be a left-invertible, non-analytic weighted shift on a rootless directed tree $\mathcal T=(V, \mathcal E)$. Then $S_{\lambdab}^\infty(\ell^2(V))$ reduces $S_{\lambdab}$ if and only if $S_{\lambdab}$ is balanced. Moreover, if $S_{\lambdab}$ is balanced, then 
%$$\ell^2(V) = %S_{\lambdab}^\infty(\ell^2(V)) \oplus \bigoplus_{n \in \mathbb N} S_{\lambdab}^n(\ker S_{\lambdab}^*).$$
$$\ell^2(V) = 
S_{\lambdab}^\infty(\ell^2(V)) \oplus \ker S_{\lambdab}^* \oplus S_{\lambdab}(\ker S_{\lambdab}^*) \oplus \cdots .$$
\end{theorem}
We begin with a computational lemma.
\begin{lemma} \label{computational} Assume that for some $m \in \mathbb Z,$ $h_{m, n} \in \ell^2(V),$ where 
\beqn
h_{m, n} := \begin{cases} \displaystyle \sum_{u \in \child{v_{m-1}}} \lambdab(u) e_{u} & ~\mbox{if~}n =1, \\
\displaystyle \sum_{u \in A(v_m,n)} \lambdab^{(n)}(u) e_{u} & ~\mbox{if~}n \Ge 2, 
\end{cases}
\eeqn
Then, for every positive integer $n,$ \beqn
S^*_{\lambdab}(h_{m, n}) =
\displaystyle \sum_{w \in A(v_{m-1}, n-1)}  \lambdab^{(n-1)}(w)\, \|S_{\lambdab}e_w\|^2\,e_w.
\eeqn
%\beqn
%S^*_{\lambdab}(h_m) &=& \begin{cases} \lambda_{v_m}e_{v_{m-1}} & \mbox{if}~n=0, \\
%\displaystyle\Big(\sum_{u \in A(v_m, 1)} \frac{\lambda^2_u}{\lambda_{v_m}}\Big)e_{v_{m-1}} & \mbox{if}~n=1, \\
%\displaystyle \sum_{w \in A(v_{m-1}, n-1)}  \frac{\lambdab^{(n-1)}(w)}{\lambdab^{(n-1)}(v_{m-1})} \frac{\|S_{\lambdab}e_w\|^2}{\lambda_{v_m}}e_w & \mbox{if}~n \Ge 2,
%\end{cases}
%\eeqn
\end{lemma}
\begin{proof} 
%If $n=0,$ then
%\beqn
%S^*_{\lambdab}(h_m)
%=S^*_{\lambdab}(e_{v_m}) = \lambda_{v_m}e_{v_{m-1}}.
%\eeqn
If $n=1,$ then
\beqn
S^*_{\lambdab}(h_{m, n})
= \sum_{u \in \child{v_{m-1}}} \lambda^2_u  e_{\parent u} = \|S_{\lambdab}e_{v_{m-1}}\|^2 e_{v_{m-1}}.
\eeqn
%which is same as $\frac{\|S_{\lambdab}e_u\|^2-\lambda^2_{v_m}}{\lambda_{v_m}} e_{v_{m-1}}.$
Assume that $n \Ge 2.$ By Lemma~\ref{preparatory}(iv), 
\beqn
S^*_{\lambdab}(h_{m, n}) &=& \sum_{u \in A(v_m,n)} \lambdab^{(n)}(u)\, \lambda_u \,e_{\parent u} \\
&=& \sum_{u \in A(v_m,n)} \lambdab^{(n-1)}(\parent u)\, \lambda^2_u \,e_{\parent u} \\
&=&  \sum_{w \in A(v_{m-1}, n-1)} \sum_{u \in \child w} \lambdab^{(n-1)}(w) \, \lambda_u^2\,e_w \\
&\overset{\eqref{known-rmk-new-new}}=& \sum_{w \in A(v_{m-1}, n-1)}  \lambdab^{(n-1)}(w)\, \|S_{\lambdab}e_w\|^2\,e_w.
\eeqn
This completes the proof. 
\end{proof}

We also need the following elementary fact.
\begin{lemma} \label{general-fact}
If $\mathcal K, \mathcal L$ are closed subspaces of a Hilbert space $\mathcal H$ such that $\mathcal K \subseteq \mathcal L,$ then 
$\mathcal H \ominus \mathcal K= \big(\mathcal H \ominus \mathcal L\big) \oplus \big(\mathcal L \ominus \mathcal K\big).$
\end{lemma}
\begin{proof} Let $h \in \mathcal H \ominus \mathcal K.$ Decompose $h= x + y,$ where $x \in \mathcal H \ominus \mathcal L$ and $y \in \mathcal L.$ Since $h$ is orthogonal to $\mathcal K,$ for any $k \in \mathcal K,$ $\inp{x+y}{k}=0.$ However, since $x$ is orthogonal to $\mathcal L,$ $\inp{y}{k}=0$ for every $k \in \mathcal K.$ This shows that $y \in \mathcal L \ominus \mathcal K,$ and hence
\beqn \mathcal H \ominus \mathcal K \subseteq \big(\mathcal H \ominus \mathcal L\big) \oplus \big(\mathcal L \ominus \mathcal K\big).
\eeqn
Since $\mathcal H \ominus \mathcal L \subseteq \mathcal H \ominus \mathcal K$ and $\mathcal L \ominus \mathcal K \subseteq \mathcal H \ominus \mathcal K,$ we obtain the other inclusion. 
\end{proof}

Here is the third lemma needed in the proof of Theorem~\ref{balanced-reducing} (cf. \cite[Theorem 6.4]{BDPP2019}).
\begin{lemma} \label{balanced-Wold}
Let $S_{\lambda}$ be a left-invertible balanced weighted shift on a rootless directed tree $\mathcal T=(V, \mathcal E)$. Then the following statements are true:
\begin{itemize}
\item[(i)] $\|S_{\lambdab}^n e_v\| = \|S_{\lambdab}^n e_u\|$ for every $n \in \mathbb N$ and for all vertices $u$, $v$ lying in the same generation,
\item[(ii)] $S_{\lambdab}^n(\ker S_{\lambdab}^*)$ and $S_{\lambdab}^m(\ker S_{\lambdab}^*)$ are orthogonal for all distinct $m,n \in \mathbb N$,
\item[(iii)] $S_{\lambdab}^n(\ker S_{\lambdab}^*) = S_{\lambdab}^n(\ell^2(V)) \ominus S_{\lambdab}^{n+1}(\ell^2(V)),$ $n \in \mathbb N$,
\item[(iv)] $\oplus_{k=0}^{n-1} S_{\lambdab}^k(\ker S_{\lambdab}^*) =  \ell^2(V) \ominus S_{\lambdab}^{n}(\ell^2(V)),$ $n \Ge 1$.
%\item[(v)] $\ell^2(V) \ominus S_{\lambdab}^\infty(\ell^2(V))= \oplus_{n \in \mathbb N} S_{\lambdab}^n(\ker S_{\lambdab}^*)$.
\end{itemize}
\end{lemma}
\begin{proof} 
Let $c(m,v)$ denote the constant value of $\|S_{\lambdab}^m e_w\|^2,$ $w \in \child v$.

(i) This can be obtained by imitating the argument of \cite[Lemma 6.3]{BDPP2019} (indeed, the proof of \cite[Lemma 6.3]{BDPP2019} therein relies only on \cite[Proposition 3.1.3]{JJS2012}, which is valid for a bounded weighted shift on arbitrary directed tree).
%\cite[Lemma 6.3 \& Theorem 6.4(iii)]{BDPP2019}. 
%The ideas involved in the proof of (iii) and (iv) are same as those appeared in the proof of Wold decomposition of isometries, see \cite[Theorem 1.1]{SF}. 

(ii) 
Let $f \in \ker S_{\lambdab}^*$. In view of \cite[Proposition~3.5.1(ii)]{JJS2012}, we may assume that $f = \sum_{w \in \child v} f(w) e_w$ for some $v \in V.$ For $m, n \in \mathbb N$ such that $m < n$ and $g \in \ell^2(V)$,
\beq
\label{ortho-step}
\inp{S_{\lambdab}^m(f)}{S_{\lambdab}^{n}(g)} &=& \inp{S_{\lambdab}^{*m}S_{\lambdab}^m(f)}{S^{n-m}_{\lambdab}(g)} \notag \\ \notag
&\overset{\eqref{known-rmk-new-new-new}}=&\left\langle \sum_{w \in \child v} \|S_{\lambdab}^m e_w\|^2 f(w) e_w, ~S^{n-m}_{\lambdab}(g)\right\rangle \\
&\overset{\mbox{(i)}}=& c(m,v) \inp{f}{S^{n-m}_{\lambdab}(g)} = 0.
\eeq
This yields (ii). 

(iiii) Letting $n=m+1$ in \eqref{ortho-step}, we obtain 
\beq\label{inclusion-1}
S_{\lambdab}^m(\ker S_{\lambdab}^*) \subseteq S_{\lambdab}^m(\ell^2(V)) \ominus S_{\lambdab}^{m+1}(\ell^2(V)).
\eeq
To see the reverse inclusion, let $f \in S_{\lambdab}^m(\ell^2(V)) \ominus S_{\lambdab}^{m+1}(\ell^2(V))$. Then $f = S_{\lambdab}^m (g)$ for some $g \in \ell^2(V)$ and $\inp{S_{\lambdab}^m (g)}{S_{\lambdab}^{m+1}(e_v)} = 0$ for all $v \in V$. However, since 
\beqn
\inp{S_{\lambdab}^m (g)}{S_{\lambdab}^{m+1} e_v} = \left\langle g,\ S_{\lambdab}^{*m}S_{\lambdab}^m\sum_{w \in \child v} \lambda_w e_w\right\rangle 
\overset{\eqref{known-rmk-new-new-new}}= c(m,v) \inp{g}{S_{\lambdab} e_v},
\eeqn
$\inp{g}{S_{\lambdab} e_v} = 0$ for every $v \in V,$ that is, $g \in \ker S_{\lambdab}^*$. Consequently, $f \in S_{\lambdab}^m(\ker S_{\lambdab}^*)$. This together with \eqref{inclusion-1} yields (iii).

(iv) We verify this by mathematical induction on $n \Ge 1.$ Clearly, (iv) is obvious for $n=1.$ Assume that (iv) holds for $n \Ge 1.$ Thus 
\beqn
%\label{inclusion-iv}
\oplus_{k=0}^{n} S_{\lambdab}^k(\ker S_{\lambdab}^*) 
&=& \Big(\ell^2(V) \ominus S_{\lambdab}^{n}(\ell^2(V))\Big) \oplus S_{\lambdab}^{n}(\ker S_{\lambdab}^*) \notag
\\
&\overset{(\mathrm{iii})}=& \Big(\ell^2(V) \ominus S_{\lambdab}^{n}(\ell^2(V))\Big) \oplus \Big(S_{\lambdab}^{n}(\ell^2(V)) \ominus S_{\lambdab}^{n+1}(\ell^2(V))\Big) 
\notag\\
& = & \ell^2(V) \ominus S_{\lambdab}^{n+1}(\ell^2(V)),
\eeqn
where we used Lemma~\ref{general-fact} in the last step. This completes the verification of (iv).
\end{proof}

\begin{proof}[Proof of Theorem~\ref{balanced-reducing}]
For a bilateral path ${\mf v}=\{v_m\}_{m \in \mathbb Z}$ in $\mathcal T,$ let $g_{m, \mf v} \in \ell^2(V),$ $m \in \mathbb Z$, be given by 
\eqref{gm-def}. Fix $m \in \mathbb Z.$ 
Since
$A(v_m, 0) \sqcup A(v_m, 1)= \child{v_{m-1}},$ 
by Lemma~\ref{computational},
\beq \label{bbws-new}
S_{\lambdab}^* g_{m, \mf v} =
\sum_{n = 1}^{\infty} \sum_{w \in A(v_{m-1}, n-1)}  \frac{\lambdab^{(n-1)}(w)}{\lambdab^{n}(v_m)}\, \|S_{\lambdab}e_w\|^2\,e_w.
\eeq 
Now, assume that $S_{\lambdab}^\infty(\ell^2(V))$ reduces $S_{\lambdab}.$ 
By Theorem~\ref{hyper-range-tree}, $S^*_{\lambdab}$ is the backward bilateral weighted shift given by
\beq 
\label{bbws}
S^*_{\lambdab}g_{m, \mf v}=\frac{\|g_{m, \mf v}\|^2}{\|g_{m-1, \mf v}\|^2}\lambda_{v_m}g_{m-1, \mf v}, \quad m \in \mathbb Z.
\eeq
This combined with \eqref{bbws-new} shows that for every $m \in \mathbb Z,$ 
\beqn
&& \sum_{n = 0}^{\infty} \sum_{w \in A(v_{m-1}, n)}  \frac{\lambdab^{(n)}(w)}{\lambdab^{(n+1)}(v_m)}\, \|S_{\lambdab}e_w\|^2\,e_w \\
&=& \frac{\|g_{m, \mf v}\|^2}{\|g_{m-1, \mf v}\|^2}\lambda^2_{v_m} \sum_{n = 0}^{\infty} \sum_{w \in A(v_{m-1},n)} \frac{\lambdab^{(n)}(w)}{\lambdab^{(n+1)}(v_{m})} e_w. 
\eeqn
Comparing the coefficients of $e_w$ on both sides, we obtain
\beqn
\|S_{\lambdab} e_{w}\|^2 = \frac{\|g_{m, \mf v}\|^2}{\|g_{m-1, \mf v}\|^2}\,\lambda^2_{v_m}, \quad w \in A(v_{m-1},n), ~n \Ge 0. 
\eeqn
It now follows from Lemma~\ref{preparatory}(i) that  $S_{\lambdab}$ is balanced. 
Conversely, if $S_{\lambdab}$ is balanced, then by \eqref{bbws-new}, $S^*_{\lambdab}g_{m, \mf v} = c_m g_{m-1, \mf v},$ $m \in \mathbb Z,$ where $c_m$ is the constant value of $\|S_{\lambdab}e_u\|,$ $u \in \mathscr G_{v_m}.$
It is now easy to see that \eqref{bbws} holds. This shows that $S_{\lambdab}^\infty(\ell^2(V))$ reduces $S_{\lambdab},$ which completes the proof of the first part.
 
To see the remaining part, note that by Lemma~\ref{balanced-Wold}(iii), for every $n \in \mathbb N,$   the closed subspace $S_{\lambdab}^\infty(\ell^2(V))$ of $S_{\lambdab}^{n+1}(\ell^2(V))$ is orthogonal to $S_{\lambdab}^n(\ker S_{\lambdab}^*).$ This together with Lemma~\ref{balanced-Wold}(ii) gives
\beqn
S_{\lambdab}^\infty(\ell^2(V)) \subseteq \ell^2(V) \ominus \Big(\oplus_{n\in\mathbb N} S_{\lambdab}^n(\ker S_{\lambdab}^*)\Big).
\eeqn
To see the reverse inclusion, note that 
%\beq\label{finite-dec}
%\ker S_{\lambdab}^* \oplus S_{\lambdab}(\ker S_{\lambdab}^*) \oplus \cdots \oplus S_{\lambdab}^{n-1}(\ker S_{\lambdab}^*) =  \ell^2(V) \ominus S_{\lambdab}^{n}(\ell^2(V)).
%\eeq
if $f \in \ell^2(V) \ominus \Big(\oplus_{n\in\mathbb N} S_{\lambdab}^n(\ker S_{\lambdab}^*)\Big),$ then for every integer $n \Ge 1,$ $f \in \ell^2(V) \ominus \Big(\oplus_{k=0}^{n-1} S_{\lambdab}^k(\ker S_{\lambdab}^*)\Big)$, and consequently, by Lemma~\ref{balanced-Wold}(iv), $f \in S_{\lambdab}^{n}(\ell^2(V)),$ and hence $f \in S_{\lambdab}^{\infty}(\ell^2(V)).$ 
This completes the proof. 
\end{proof}

\begin{proof}[Proof of Theorem~\ref{Wold-type-thm}]
We divide the proof into two cases:
\vskip.1cm
{\it Case 1. $S_{\lambdab}$ is analytic.} 
\vskip.1cm
By Corollary~\ref{dicho-coro}, $\alpha_{\lambdab}(v)=\infty$ for every $v \in V.$ To see the desired equivalence, 
note that $S_{\lambdab}$ has Wold-type decomposition if and only if $S_{\lambdab}$ has the wandering subspace property. In view of \cite[Proposition~2.7]{S2001}, this happens if and only if $S'_{\lambdab}$ is analytic. It now follows from Corollary~\ref{dicho-coro} that $S_{\lambdab}$ has Wold-type decomposition if and only if  $\alpha_{\lambdab'}(v) = \infty$ for every $v \in V.$
\vskip.1cm
{\it Case 2. $S_{\lambdab}$ is non-analytic.}
\vskip.1cm
By Corollary~\ref{dicho-coro}, $\alpha_{\lambdab}(v) < \infty$ for every $v \in V.$ 
Also, by Theorem~\ref{hyper-range-tree}, $S_{\lambdab}\, g_{m, \mf v} = \lambda_{v_{m+1}} g_{m+1, \mf v},$ $m \in \mathbb Z,$ where $g_{m, \mf v}$ given by \eqref{gm-def} forms an orthogonal basis for $S_{\lambdab}^\infty(\ell^2(V)).$  

Suppose that $S_{\lambdab}$ has Wold-type decomposition.
Since $S_{\lambdab}|_{S_{\lambdab}^\infty(\ell^2(V))}$ is an isometry, $\lambda_{v_{m+1}} = \sqrt{\frac{\alpha_{\lambdab}(v_m)}{\alpha_{\lambdab}(v_{m+1})}}$ for every $m \in \mathbb Z.$ Since Theorem~\ref{hyper-range-tree} is applicable to any path $\{v_m\}_{m \in \mathbb Z}$ with arbitrary $v_0 \in V,$ 
\beqn 
\mbox{$S_{\lambdab}|_{S_{\lambdab}^\infty(\ell^2(V))}$ is an isometry}~\Longleftrightarrow~\mbox{
$\lambda_{v} =\sqrt{ \frac{\alpha_{\lambdab}(\parent{v})}{\alpha_{\lambdab}(v)}}$ for every $v \in V.$}
\eeqn
Further, by Theorem~\ref{balanced-reducing}, 
\beqn
\mbox{$S_{\lambdab}^\infty(\ell^2(V))$ reduces $S_{\lambdab}$}~\Longleftrightarrow~\mbox{$S_{\lambdab}$ is balanced.}
\eeqn 
This gives part (ii) of Theorem~\ref{Wold-type-thm}. 

Suppose now that $\lambda_{v} =\sqrt{\frac{\alpha_{\lambdab}(\parent{v})}{\alpha_{\lambdab}(v)}}$ for all $v$ in a bilateral path, and $S_{\lambdab}$ is balanced. By Theorem~\ref{balanced-reducing}, $S_{\lambdab}^\infty(\ell^2(V))$ reduces $S_{\lambdab}$ and 
$S_{\lambdab}|_{_{\ell^2(V) \ominus S_{\lambdab}^\infty(\ell^2(V))}}$ has wandering subspace property.
Also, by Theorem~\ref{hyper-range-tree} that $S_{\lambdab}|_{S_{\lambdab}^\infty(\ell^2(V))}$ is unitary. This completes the proof.   
\end{proof}

As a consequence, we obtain a variant of \cite[Proposition 5.11]{ACJS2019}.
\begin{corollary} Assume that $S_{\lambdab} \in \mathcal B(\ell^2(V))$ is a norm-increasing $m$-concave weighted shift on a rootless directed tree $\mathcal T = (V, \mathcal E).$ If the cardinality of a generation $\mathscr G_v$ is finite for some $v \in V$, then $S_{\lambdab}$ is an isometry.
%If $S_{\lambdab}$ non-analytic, then it is an isometry.
\end{corollary}
\begin{proof} Suppose that $\mbox{card}(\mathscr G_v) < \infty$ for some $v \in V.$  
By Remark~\ref{remark-alpha-finite}, $\alpha_{\lambdab}(v) < \infty.$ Hence, by Theorem~\ref{hyper-range-tree}, 
$S_{\lambdab}$ is non-analytic. By \cite[Proposition 3.4]{S2001} (applied to $T=S'_{\lambdab}$),  $S_{\lambdab}^\infty(\ell^2(V))$ reduces $S_{\lambdab}.$ 
It now follows from 
Theorem~\ref{balanced-reducing} that $S_{\lambdab}$ is balanced. Let $\{v_n\}_{n \in \mathbb Z}$ be a bilateral path and set $\gamma_n = \|S_{\lambdab}e_{v_n}\|,$ $n \in \mathbb Z.$ Consider weighted shift $S_{\gammab}$ on rootless directed tree $\mathbb Z$ (with standard directed tree structure). Using
%given by
\beqn
S_{\gammab}e_n = \gamma_n e_{n+1}, \quad n \in \mathbb Z,
\eeqn
%where $\{e_n\}_{n \in \mathbb Z}$ is the standard orthonormal basis for $\ell^2(\mathbb Z).$ 
it is easy to see that $S_{\gammab}$ is an invertible norm-increasing $m$-concave operator. It follows from \cite[Proposition 3.4]{S2001} that $S_{\gammab}$ is unitary, and hence $\|S_{\lambdab}e_{v_n}\|=1$ for every $n \in \mathbb Z.$ Since $\{v_n\}_{n \in \mathbb Z}$ is an arbitrary bilateral path, $\|S_{\lambdab}e_{v}\|=1$ for every $v \in V.$ This combined with \eqref{known-rmk-new-new-new} shows that $S_{\lambdab}$ is an isometry.
\end{proof}

%\begin{corollary}
%Let $S_{\lambdab}$ be a left-invertible weighted shift on a rootless directed tree $\mathcal T = (V, \mathcal E)$. If the cardinality of a generation $\mathscr G_v$ is finite for some $v \in V$, then $S_{\lambdab}$ has the Wold-type decomposition
%\beqn
%S_{\lambdab} = U \oplus W~\mbox{on~}\ell^2(V) = S^{\infty}_{\lambdab}(\ell^2(V)) \oplus 
%[\ker S^*_{\lambdab}]_{S_{\lambdab}}, 
%\eeqn
%where $U$ is an unitary and $W$ is an operator-valued weighted shift. 
%\end{corollary}
%\begin{proof}
%Suppose that the cardinality of a generation $\mathscr G_v$ is finite for some $v \in V$. By Remark~\ref{remark-alpha-finite}, $\alpha_{\lambdab}(v) < \infty.$ Hence, Theorem~\ref{Wold-type-thm}(ii) holds. Thus it suffices to check that the restriction of $S_{\lambdab}$ to $[\ker S^*_{\lambdab}]_{S_{\lambdab}}$ is an operator-valued weighted shift. 
%\end{proof}

\section{Examples} \label{S5}

%Let us understand \eqref{A-v-n-def} and \eqref{v-m-def} with the help of an example. 
The purpose of this section is to illustrate the utility of Theorem~\ref{Wold-type-thm} with the help of a family of weighted shifts on a particular kind of rootless directed trees introduced in \cite[Example 4.4]{J2003}. To see that, 
consider the rootless 
quasi-Brownian directed tree of valency $2$. This directed tree, denoted by $\mathcal T_{qb},$ is given by
\begin{align*}
 & V = \mathbb N \times \mathbb Z, & \\ 
 & \mathcal E = \big\{\big((0, m), (1, m)\big), \big((0, m), (0, m-1)\big) : m \in \mathbb Z\big\}& \\
 & \sqcup  \big\{\big((n, m), (n+1, m)\big) : n \Ge 1, m \in \mathbb Z\big\}.&
\end{align*} 
Geometrically, $\mathcal T_{qb}$ is
obtained by ``gluing” a copy of the directed tree $\mathbb N$ to each $m \in \mathbb Z$ (see Figure~\ref{Fig-gen}). 
Note that $\mathcal T_{qb}$ is a leafless directed tree, which is not of finite branching index. This follows from 
\allowdisplaybreaks
\beq
\label{edge-set}
\child{(n,m)} = \begin{cases}
\{(n,m-1), (1,m)\} & \mbox{ if } n=0,\\
\{(n+1,m)\} & \mbox{ if } n \geqslant 1. \end{cases}
\eeq
This, combined with an induction on $n \Ge 1,$ yields 
%\begin{align} 
\beqn
%\label{childn-rmk}
\left. \arraycolsep=1pt\def\arraystretch{1.4}
 \begin{array}{cc}
 &  \childn{n}{(0,m)} = \{(j,m-n+j) : 0 \leqslant j \leqslant n\},  \\ 
 &  \childn{n}{(l,m)} = \{(l+n,m)\},  
 \end{array}
\right\}
m \in \mathbb Z, ~ n, l \geqslant 1.
%\end{align}
\eeqn
It now follows from \eqref{known-rmk-new} that for every integer $k \Ge 1,$ 
\beq 
\label{orthogonality}
&& \mbox{the set $\{S^k_{\lambdab} e_{(n, m)} : n \in \mathbb N, m \in \mathbb Z\}$ is orthogonal in $\ell^2(V)$}.
\eeq
%\begin{figure}[H]
%\includegraphics[scale=.7]{Fig1.pdf}\caption{The 
%quasi-Brownian directed tree $\mathcal T_{qb}$ of valency $2$}\label{Fignew}
%\end{figure}
Let $v_m:=(0,m)$, $m \in \mathbb Z$. Then the generation of $v_m$ is given by 
%\beq \label{Am-def}
$\mathscr G_{v_m} = \{(j,m+j) : j \in \mathbb N\}.$
%\eeq
Further, for all $n \geqslant 1$ and $m \in \mathbb Z$,
$$\childn{n}{\parentn{n}{v_m}} = \{(j,m-j) : 0 \leqslant j \leqslant n\}.$$
Thus, $A(v_m, n) = \{(n,m-n)\}$
$($see Figure \ref{Fig-gen}$).$
\begin{figure}[H]
\includegraphics[scale=.7]{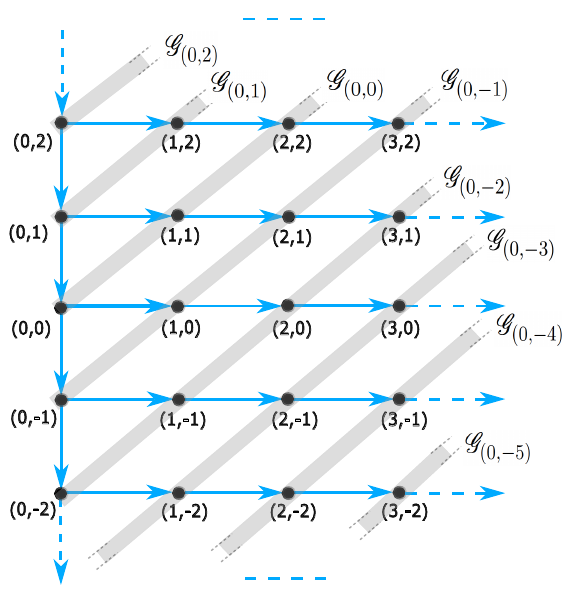}\caption{The generations of $\mathcal T_{qb}$}\label{Fig-gen}
\end{figure}

\begin{proposition} \label{prop-exam}
Let $\{p_m\}_{m \in \mathbb Z}$ be a bi-infinite sequence of quadratic polynomials  with positive coefficients such that  
\allowdisplaybreaks
\beq \label{constant-1}
& p_m(0)=1, ~m \in \mathbb Z, \\ \label{bdd-wt-shift}
& \displaystyle \sup_{m \in \mathbb Z} \big\{p'_m(0), p''_m(0)\big\} < \infty, \\ \label{inf-posi}
& \displaystyle \inf_{m \in \mathbb Z} p_m''(0) > 0.
\eeq
Consider the weight system $\lambdab = \big\{\lambda_{(n, m)} : (n, m) \in \mathbb N \times \mathbb Z\big\}$ given by 
\begin{align} \label{wt-system}
& \lambda_{(n,m)} := \begin{cases} \sqrt{\frac{p_m(n-1)}{p_m(n-2)}} & \mbox{ if } m \in \mathbb Z, ~n \geqslant 2, \\
                             \sqrt{\frac{m}{m+1}} & \mbox{ if } m \geqslant 1, ~n=0,\\
                             \frac{1}{\sqrt m} & \mbox{if} ~m \geqslant 1,~n=1,\\
                             1 & \mbox{if} ~m < 1,~n=0, 1. \end{cases}
\end{align}
Then $S_{\lambdab} \in \mathcal B(\ell^2(\mathbb N \times \mathbb Z))$ and $S_{\lambdab}$ is an analytic norm-increasing operator without Wold-type decomposition.  
%Then the following statements hold:
%\begin{itemize}
%\item[(i)] {\color{red}If for all $m \in \mathbb Z$,
%\beq\label{basic-c}
% \quad\ p'_m(0) - p'_{m-1}(0) + \delta_{2}(m)= \frac{1}{2}\Big(p''_m(0) + p''_{m-1}(0)\Big) + \delta_{3}(m),
%\eeq
%then $S_{\lambdab}$ is an analytic norm-increasing $3$-isometry with the wandering subspace property.}
%\item[(ii)] If $\inf_{m \in \mathbb Z} p_m''(0) > 0$ and for all $m \in \mathbb Z$,
%\beq\label{basic-c1}
%\quad\ p'_m(0) - p'_{m-1}(0) + \delta_{2}(m) < \frac{1}{2}\Big(p''_m(0) + p''_{m-1}(0)\Big) + \delta_{3}(m),
%\eeq
%then $S_{\lambdab}$ is an analytic norm-increasing $3$-expansion without the wandering subspace property.
%\end{itemize}
\end{proposition}
\begin{proof}
%[Proof of Proposition~\ref{prop-exam}]
We first check that $S_{\lambdab} \in \mathcal B(\ell^2(\mathbb N \times \mathbb Z)).$ 
To see this, by \cite[Proposition~3.1.8]{JJS2012}, it suffices to check that
\beqn
\sup_{(n, m) \in \mathbb N \times \mathbb Z} \sum_{(k, l) \in \child{(n, m)}} \lambda^2_{(k, l)} < \infty.
\eeqn
In turn, in view of \eqref{edge-set} and \eqref{wt-system}, it is enough to show that 
\beq
\label{bdd-poly-reci}
\sup_{m \in \mathbb Z, \, n \Ge 2}\frac{p_m(n-1)}{p_m(n-2)} < \infty.
\eeq
It is easy to see that $p_m(n-1) \Le 4 p_m(n-2)$ for every $m \in \mathbb Z$ and $n \Ge  3.$ Also, by \eqref{constant-1} and \eqref{bdd-wt-shift},
$\{p_m(1)\}_{m \in \mathbb Z}$ is bounded. This together with \eqref{constant-1} yields 
\eqref{bdd-poly-reci}, completing the verification of $S_{\lambdab} \in \mathcal B(\ell^2(\mathbb N \times \mathbb Z)).$ 
 
It is easy to see using \eqref{constant-1} and \eqref{wt-system} that
\allowdisplaybreaks
\beq 
\label{eq1}
\|S_{\lambdab} e_{(0,m)}\|^2 &=& 
%\lambda^2_{(0,m-1)} + \lambda^2_{(1,m)} =
\begin{cases} 1 & \mbox{if}~m \Ge 2, \\
2 & \mbox{if}~m \Le 1.
\end{cases}
\eeq 
A routine calculation shows that 
\beq 
\label{norm-above-0}
&& \|S_{\lambdab}^k e_{(n,m)}\|^2 = \frac{p_m(n+k-1)}{p_m(n-1)}, \quad n \geqslant 1, m \in \mathbb Z.
\eeq
Since $\|S_{\lambdab} e_{(n,m)}\|^2 = \frac{p_m(n)}{p_m(n-1)},$ $n \geqslant 1,$ $m \in \mathbb Z,$ it is now clear from \eqref{orthogonality} that $S_{\lambdab}$ is norm-increasing. 

Next, we check that $S_{\lambdab}$ is analytic. 
By \eqref{wt-system}, $\lambda_{(0, m)} \Le 1$ for any $m \in \mathbb Z.$  Another application of \eqref{wt-system} shows that 
\allowdisplaybreaks
\beqn
\alpha_{\lambdab}\big((0,m)\big) = \sum_{n = 0}^{\infty} \Big(\prod_{j=1}^n \frac{\lambda_{(j,m+n)}}{ \lambda_{(0,m+j-1)}}\Big)^2  \Ge  \begin{cases} \displaystyle \sum_{n = 1}^{\infty} \frac{p_{m+n}(n-1)}{m+n} 
& \mathrm{if}~ m \Ge 1, \\
\displaystyle \sum_{n = -m+1}^{\infty}  \frac{p_{m+n}(n-1)}{m+n} & \mathrm{if}~m \Le 0.
\end{cases}
\eeqn 
This, together with the fact that $p_m(n) \Ge 1$ for every $n \in \mathbb N$ and $m \in \mathbb Z$ (see \eqref{constant-1}), shows that both the series on the right-hand side above are divergent. 
Hence, by the virtue of Corollary~\ref{dicho-coro}, $S_{\lambdab}$ is analytic.

Now we show that $\alpha_{\lambdab'}\big((0,0)\big) < \infty$. Note that
\beqn 
%\label{wt-dual}
\lambda'_{(n,m)} &=& \begin{cases} \frac{\lambda_{(0,m)}}{\lambda^2_{(0,m)} + \lambda^2_{(1, m+1)}} = \sqrt{\frac{m}{m+1}} & \mbox{if} ~n=0, m \Ge 1,\\
\frac{\lambda_{(1,m)}}{\lambda^2_{(0,m-1)} + \lambda^2_{(1, m)}} = \frac{1}{\sqrt{m}} & \mbox{if} ~n=1, m \Ge 2,\\
\frac{1}{\lambda_{(n,m)}} = \sqrt{\frac{p_m(n-2)}{p_m(n-1)}}& \mbox{if} ~n \Ge 2, m \Ge 1.
\end{cases}
\eeqn
It is now easy to see that 
\beqn
%\label{cgn-3-iso}
\alpha_{\lambdab'}\big((0,0)\big) &=& \sum_{n = 1}^{\infty} \Big(\prod_{j=1}^n \frac{\lambda'_{(j,n+1)}}{ \lambda'_{(0,j)}}\Big)^2 
%= \sum_{n = 0}^{\infty} \Big(\frac{\prod_{j=2}^n \lambda'_{(j,n+1)}}{\prod_{j=1}^n \lambda'_{(0,j)}}\Big)^2
= \sum_{n = 1}^{\infty} \frac{1}{p_{n+1}(n-1)}.
%\\
%&=& \sum_{n = 1}^{\infty} \frac{1}{1+ a_{n+1}(n-1) + b_{n+1}(n-1)^2},
\eeqn 
%where $p_m(n) = 1 + a_m n + b_m n^2$, $m \in \mathbb Z$ and $n \in \mathbb N$. 
Since $p_m(0)=1,$ $c = \inf_{m \in \mathbb Z} p_m''(0) > 0$  and $p'_m(0) > 0$,   
\beqn
\alpha_{\lambdab'}\big((0,0)\big) 
%&=& \sum_{n = 1}^{\infty} \Big(\prod_{j=1}^n \frac{\lambda'_{(j,n+1)}}{ \lambda'_{(0,j)}}\Big)^2 
%= \sum_{n = 0}^{\infty} \Big(\frac{\prod_{j=2}^n \lambda'_{(j,n+1)}}{\prod_{j=1}^n \lambda'_{(0,j)}}\Big)^2
%= \sum_{n = 1}^{\infty} \frac{1}{p_{n+1}(n-1)}\notag\\
%&=& \sum_{n = 1}^{\infty} \frac{1}{1+ a_{n+1}(n-1) + b_{n+1}(n-1)^2}\\
\Le \sum_{n = 1}^{\infty} \frac{2}{2+ c (n-1)^2} < \infty.
\eeqn
Hence, by Theorem~\ref{Wold-type-thm}, $S_{\lambdab}$ does not have Wold-type decomposition. 
%{\color{red}Thus, $S_{\lambdab'}$ is not analytic, and hence, $S_{\lambdab}$ does not have the wandering subspace property.}
\end{proof}

If we consider $S_{\lambdab}$ as in Proposition~\ref{prop-exam}, then $S_{\lambdab}$ is never $3$-concave. However, we are able to exhibit an analytic norm-increasing $3$-expansive $S_{\lambdab}$ without Wold-type decomposition.
\begin{example} \label{Exam-ce}
Let $\{p_m\}_{m \in \mathbb Z}$ be a bi-infinite sequence of quadratic polynomials given by $p_m(n) = 1 + a_m n + b_m n^2$, $m \in \mathbb Z$ and $n \in \mathbb N$ with positive coefficients $a_m$ and $b_m.$ Assume that $\sup_{m \in \mathbb Z}a_m < \infty,$ $\sup_{m \in \mathbb Z}b_m < \infty$ and $\inf_{m \in \mathbb Z}b_m > 0.$ Note that   \eqref{constant-1}-\eqref{inf-posi} hold. 
In addition, assume that 
%\beq\label{basic-c1}
%\quad\ p'_m(0) - p'_{m-1}(0) + \delta_{2}(m) < \frac{1}{2}\Big(p''_m(0) + p''_{m-1}(0)\Big) + \delta_{3}(m),
%\eeq
\beq\label{basic-c1}
\quad a_m - a_{m-1} + \delta_{2}(m) \Le b_m + b_{m-1} + \delta_{3}(m), \quad m \in \mathbb Z,
\eeq
where $\delta_n$ denote the Dirac delta function with a point mass at $n \in \mathbb Z.$ 
We check that $S_{\lambdab}$ is a $3$-expansion. 
By \eqref{orthogonality} and the fact that $S_{\lambdab}$ is a bounded linear operator on $\ell^2(\mathbb N \times \mathbb Z),$  
$S_{\lambdab}$ is a $3$-expansion if and only if for every $(n,m) \in \mathbb N \times \mathbb Z,$
\beq \label{3-iso-equivalent}
1 - 3\|S_{\lambdab} e_{(n,m)}\|^2 + 3 \|S_{\lambdab}^2 e_{(n,m)}\|^2 - \|S_{\lambdab}^3 e_{(n,m)}\|^2 \Le 0, \quad 
\eeq
By \eqref{norm-above-0}, $\|S_{\lambdab}^k e_{(n,m)}\|^2$ is a quadratic polynomial in $k$. 
Hence,  
\eqref{3-iso-equivalent} holds (with equality) for every integer $n \geqslant 1$ and $m \in \mathbb Z$ (see, for instance,  \cite[Theorem~8.4]{D1967}). 
To complete the verification of \eqref{3-iso-equivalent}, 
fix $m \in \mathbb Z$ and note that 
by \eqref{edge-set} and \eqref{known-rmk-new}, 
\allowdisplaybreaks
\beqn
S^j_{\lambdab} e_{(0,m)} = 
\lambda_{(0, m-1)}S^{j-1}_{\lambdab} e_{(0,m-1)} +  \prod_{k=1}^j \lambda_{(k, m)} e_{(j, m)}, \quad j=1, 2, 3. 
\eeqn
This implies that the vectors $S^{j-1}_{\lambdab} e_{(0,m-1)}$ and $e_{(j, m)}$ are orthogonal, and hence for any $j=1, 2, 3,$ we obtain 
\allowdisplaybreaks
\beqn
&& \|S^j_{\lambdab} e_{(0,m)}\|^2 = 
\lambda^2_{(0, m-1)}\|S^{j-1}_{\lambdab} e_{(0,m-1)}\|^2 +  \prod_{k=1}^j \lambda^2_{(k, m)}, \\
&\overset{\eqref{constant-1}\&\eqref{wt-system}}=& \begin{cases} \frac{1}{m}\big((m-1)\|S^{j-1}_{\lambdab} e_{(0,m-1)}\|^2 + p_m(j-1)\big)  & \mbox{if}~m \Ge 2, \\
\|S^{j-1}_{\lambdab} e_{(0,m-1)}\|^2 + p_m(j-1), & \mbox{if}~m \Le 1. 
\end{cases}   
\eeqn 
This together with \eqref{constant-1} and \eqref{wt-system} yields 
\allowdisplaybreaks
\beq 
\label{eq2}
\|S^2_{\lambdab} e_{(0,m)}\|^2 
%&=&  
%\lambda^2_{(0, m-1)}\|S_{\lambdab} e_{(0,m-1)}\|^2 +  \lambda^2_{(1, m)} \lambda^2_{(2, m)} \notag \\
&\overset{\eqref{eq1}}=& \!\!\! \begin{cases} \frac{1}{m}\big(m-1 + p_m(1)\big) & \mbox{if}~m \Ge 3, \\
\frac{1}{2}\big(2 + p_2(1)\big) & \mbox{if}~m = 2, \\
2 + p_m(1) & \mbox{if}~m \Le 1, 
\end{cases}
\\ 
\label{eq3}
\|S^3_{\lambdab} e_{(0,m)}\|^2 
%&=& 
%\lambda^2_{(0, m-1)}\|S^2_{\lambdab} e_{(0,m-1)}\|^2 +  \lambda^2_{(1, m)} \lambda^2_{(2, m)} \lambda^2_{(3, m)} \notag \\ 
&\overset{\eqref{eq2}}=& \!\!\! \begin{cases} \frac{1}{m} \big(m-2 + p_{m-1}(1) + p_{m}(2)\big) & \mbox{if}~m \Ge 4, \\
\frac{1}{3} \big(2 + p_{2}(1) + p_{3}(2)\big) & \mbox{if}~m = 3, \\
\frac{1}{2}\big(2 + p_{1}(1) + p_{2}(2)\big) 
& \mbox{if}~m = 2, \\
2 + p_{m-1}(1) + p_{m}(2) & \mbox{if}~m \Le 1.
\end{cases}
\eeq   
Thus \eqref{eq2} and \eqref{eq3} shows that 
\beqn && 1 - 3\|S_{\lambdab} e_{(0,m)}\|^2 + 3 \|S_{\lambdab}^2 e_{(0,m)}\|^2 - \|S_{\lambdab}^3 e_{(0,m)}\|^2 \\
&=& \begin{cases} 
\frac{1}{m}(a_{m} - a_{m-1} - b_m - b_{m-1}) & \mbox{if~}m \Ge 4, \\
\frac{1}{3}(a_3-a_2-b_3-b_2-1) & \mbox{if~}m = 3, \\
\frac{1}{2}(a_2-a_1+1-b_2-b_1) & \mbox{if~}m = 2, \\
a_m-a_{m-1}-b_m-b_{m-1} & \mbox{if~}m \Le 1.
\end{cases}
%\Le 0, \quad m \in \mathbb Z.
\eeqn
Hence, by \eqref{basic-c1}, $S_{\lambdab}$ is a $3$-expansion. Finally, note that $p_m(n) = 1 + n + n^2$, $m \in \mathbb Z$ and $n \in \mathbb N$, satisfies \eqref{constant-1}-\eqref{inf-posi} and \eqref{basic-c1}. Consequently, with this choice of $p,$ the weights $\lambdab$ given by \eqref{wt-system} gives an analytic norm-increasing $3$-expansion $S_{\lambdab}$, which does not have Wold-type decomposition. 
\eof
\end{example}

\section*{Appendix: Invariant subspaces and Weyl sequences}
Theorem~\ref{dichotomy-new-coro} turns out to be a special case of the following general fact$:$
\begin{theorem} \label{dichotomy-new}  
Let $T \in \mathcal B(\mathcal H)$ be an analytic norm-increasing operator with spectrum contained in the closed unit disc in $\mathbb C.$ 
If $\mathcal M$ is a nonzero invariant subspace of $T',$ then for any $\lambda \in \sigma_{ap}(T'|_{\mathcal M}),$ there exists a Weyl sequence in $\mathcal M$ for both $(\lambda, T)$ and $(\lambda, T').$  
\end{theorem}

Recall the fact that the fixed points of a contraction and its adjoint are the same (see \cite[Proposition 3.1]{SF}). Here is an ``approximate'' analog of this result.
\begin{lemma}\label{app-NF}
Let $S\in \mathcal B(\mathcal H)$ be a contraction. Assume that there exists a sequence $\{h_n\}_{n \Ge 1}$ of unit vectors in $\mathcal H$ such that $Sh_n- h_n \rar 0$ as $n \rar \infty.$ Then $S^*h_n- h_n \rar 0$ as $n \rar \infty.$
\end{lemma}
\begin{proof}
Since $S^*$ is a contraction and $\|h_n\|=1,$ 
\beqn
\|S^*h_n- h_n\|^2 &=& \|S^*h_n\|^2 + \|h_n\|^2 - 2 \Re(\inp{S^*h_n}{h_n}) \\
& \Le & 2(1- \Re(\inp{h_n}{Sh_n}) \\
&=& -2\Re(\inp{h_n}{Sh_n-h_n}.
\eeqn
However, $|\inp{h_n}{Sh_n-h_n}| \leq \|Sh_n- h_n\| \rar 0,$  and hence 
$\|S^*h_n- h_n\| \rar 0$ as $n \rar \infty.$
This completes the proof. 
\end{proof}

All nontrivial approximate eigenvalues yield Weyl sequences. 
\begin{lemma}\label{lemma-ap-p}
Let $S\in \mathcal B(\mathcal H)$ and let $\lambda \in \mathbb C \backslash \sigma_{p}(S)$ be nonzero. Assume that there exists a sequence $\{h_n\}_{n \Ge 1}$ of unit vectors in $\mathcal H$ such that  
\beq \label{S-app-0}
\mbox{$(S-\lambda)h_n \rar 0$,}
\eeq
Then there exists a weakly null subsequence $\{h_{n_k}\}_{n \Ge 1}$ of $\{h_n\}_{n \Ge 1}.$ 
\end{lemma}
\begin{proof}  
Since $\{h_n\}_{n \Ge 1}$ is a bounded sequence in a Hilbert space, by Banach's theorem, there exists a subsequence $\{h_{n_k}\}_{k \Ge 1}$ of $\{h_n\}_{n \Ge 1}$ weakly convergent to, say, $g_0 \in \mathcal H.$ By \eqref{S-app-0} or otherwise, $\{Sh_{n_k}\}_{k \Ge 1}$ converges weakly to, $Sg_0.$ By the uniqueness of the weak limit and \eqref{S-app-0}, $Sg_0 = \lambda g_0.$ Since $\lambda \notin \sigma_p(S),$ $g_0=0.$ This shows that  $\{h_{n_k}\}_{k \Ge 1}$ is weakly null.  
\end{proof}

\begin{proof}[Proof of Theorem~\ref{dichotomy-new}]  
Since $T'$ is a left-invertible, so is $T'|_{\mathcal M}.$ Let $\lambda \in \sigma_{ap}(T'|_{\mathcal M}).$ Then $\lambda \neq 0$ and 
there exists a sequence $\{h_n\}_{n \Ge 1}$ of unit vectors in $\mathcal M$ such that 
\beq \label{T-prime-a}
\mbox{$T' h_n-\lambda h_n \rar 0$ as $n \rar \infty.$}
\eeq
However, $T'$ is a contraction (since $T$ is an expansion), and hence $|\lambda| \leqslant 1.$ As $T^*T'=I$, we must have $\lambda T^*h_n - h_n \rar 0$ as $n \rar \infty.$
Since the spectrum of $T$ is contained in the closed unit disc, $|\lambda|=1$. 
In view of \eqref{T-prime-a}, applying Lemma~\ref{app-NF} to the contractive operator $\overline{\lambda}T'$ shows that $T'^*h_n-\overline{\lambda}h_n \rar 0$ as $n \rar \infty.$ It follows that  
$T'^* T' h_n - h_n \rar 0$ as $n \rar \infty.$
Equivalently, 
$(T'^* T')^{-1} h_n - h_n \rar 0$ as $n \rar \infty.$ Since $T=(T')',$ it now follows that
\beqn
(T-\lambda)h_n &=& (T'(T'^* T')^{-1}- \lambda) h_n \\
&=& T'((T'^* T')^{-1}- I)h_n + (T'-\lambda) h_n,
\eeqn
which, by \eqref{T-prime-a}, converges to $0$ as $n \rar \infty$. Thus $\lambda \in \sigma_{ap}(T)\setminus \{0\}.$
Since an analytic operator does not have a nonzero eigenvalue, by Lemma~\ref{lemma-ap-p}, there exists a subsequence $\{h_{n_k}\}_{n \Ge 1}$ of $\{h_n\}_{n \Ge 1},$ which is a Weyl sequence for $(\lambda, T).$ It is clear from \eqref{T-prime-a} that $\{h_{n_k}\}_{k \Ge 1}$ is also a  Weyl sequence for $(\lambda, T').$ This completes the proof. 
\end{proof}

\begin{proof}[Proof of Theorem~\ref{dichotomy-new-coro}]
%[Proof of Corollary~\ref{dichotomy-new-coro}]
Since the spectrum of a norm-increasing $m$-concave operator is contained in the closed unit disc, one may apply Theorem~\ref{dichotomy-new} to $\mathcal M:=\mathcal H \ominus [\ker T^*]_T.$
\end{proof}
Let $T \in \mathcal B(\mathcal H)$ be a norm-increasing $3$-isometry. If $T$ does not have the wandering subspace property, then $T^*T-I$ has infinite rank. This follows from \cite[Corollary 8.6]{LGR2021} and the fact that an analytic $2$-isometry has the wandering subspace property (see \cite{R1988}). 

%\medskip \textit{Acknowledgment}. 


\begin{thebibliography}{100}

%\addtolength{\leftmargin}{.2in} % sets up alignment with the following line.
%\setlength{\itemindent}{-0.2in}




\bibitem{AS1995} 
J. Agler, M. Stankus, $m$-isometric transformations of Hilbert spaces, I, II, III, {\it Integr. Equ. Oper.
Theory} {\bf 21, 23, 24} (1995, 1995, 1996), 383-429, 1-48, 379-421. 
\doi{10.1007/BF01222016}



%\bibitem{ARS1996} A. Aleman, S. Richter, C. Sundberg, 
%Beurling's theorem for the Bergman space, 
%{\it Acta Math.} {\bf 177} (1996), 275-310. \doi{10.1007/BF02392623}

\bibitem{ACJS2019} A. Anand, S. Chavan, Z. J. Jabłoński, J. Stochel, A solution to the Cauchy dual subnormality problem for 2-isometries, {\it J. Funct. Anal.} {\bf 277} (2019), art. 108292, 51 pp. 
\doi{10.1016/j.jfa.2019.108292}

\bibitem{ACJS2019-2} A. Anand, S. Chavan, Z. J. Jabłoński, J. Stochel, Complete systems of unitary invariants for some classes of 2-isometries, {\it Banach J. Math. Anal.} {\bf 13} (2019), 359-385. \doi{10.1215/17358787-2018-0042}

\bibitem{ACT2020}
A. Anand, S. Chavan, S. Trivedi, Analytic $m$-isometries without the wandering subspace property, {\it Proc. Amer. Math. Soc.} {\bf 148} (2020), 2129-2142. 
\doi{10.1090/proc/14894}

%\bibitem{B1949} 
%A. Beurling, On two problems concerning linear transformations in Hilbert space, {\it Acta Math.} {\bf 81} (1949),
%239-255. \doi{10.1007/BF02395019}

\bibitem{BDPP2019}
P. Budzy\'nski, P. Dymek, A. P\l aneta, and M. Ptak, Weighted shifts on directed trees: their multiplier algebras, reflexivity and decompositions, {\it Studia Math}., {\bf 244} (2019), 285-308. \doi{10.4064/sm170220-20-9}

\bibitem{BJJS1}
P. Budzy\'nski, Z. Jab{\l}o\'nski, I.B. Jung, J. Stochel, Unbounded subnormal composition operators in $L^2$-spaces, {\it J. Funct. Anal.} {\bf 269} (2015), 2110–2164. \doi{10.1016/j.jfa.2015.01.024}

\bibitem{BJJS2017}
P. Budzy\'nski, Z. J. Jab{\l}o\'nski, I. B. Jung, J. Stochel, Subnormality of unbounded composition operators over one-circuit directed graphs: exotic examples, {\it Adv. Math.} {\bf 310} (2017), 484-556. 
\doi{10.1016/j.aim.2017.02.004}

\bibitem{CPT2017} S. Chavan, D. K. Pradhan, S. Trivedi, Multishifts on directed Cartesian products of rooted directed trees, {\it Dissertationes Math.} {\bf 527} (2017), 102 pp. \doi{10.4064/dm758-6-2017}

%\bibitem{CS2022} S. Chavan, J. Stochel, Weakly concave operators, {\it Proc. Roy. Soc. Edinburgh Sect. A: Mathematics} (2023), 1-32.  \doi{10.1017/prm.2022.85}

\bibitem{CT2016}
S. Chavan, S. Trivedi, An analytic model for left-invertible
weighted shifts on directed trees, {\it J. London Math. Soc.} {\bf 94} (2016), 253-279. 
\doi{10.1112/jlms/jdw029}

%\bibitem{CHL2021} R. E. Curto, I. S. Hwang, W. Y. Lee, 
%The Beurling-Lax-Halmos theorem for infinite multiplicity,  {\it J. Funct. Anal.} {\bf 280} (2021), Paper No. 108884, 101 pp.\doi{10.1016/j.jfa.2020.108884}

\bibitem{D1967} D. R. Dickinson, {\it Operators: An algebraic synthesis}, Macmillan, London-Melbourne-Toronto, Ont.; St. Martin's Press, New York 1967 v+245 pp.

%\bibitem{GKT}
%R. Gupta, S. Kumar and S. Trivedi, Unitary equivalence of operator-valued multishifts, {\it J. Math. Anal. Appl.} {\bf 487} (2020), 23 pp.

%\bibitem{GGR2022}
%S. Ghara, R. Gupta, Md. R. Reza, Analytic $m$-isometries and weighted Dirichlet-type spaces, {\it J. Operator Theory}, {\bf 88} (2022), 445-477. \doi{10.7900/jot.2021mar26.2335}

\bibitem{H1961} 
P. R. Halmos, Shifts on Hilbert spaces, {\it J. Reine Angew. Math.} {\bf 208} (1961), 102-112. 
\doi{10.1515/crll.1961.208.102}


%\bibitem{HKZ2000}  H. Hedenmalm, B. Korenblum, K. Zhu, {\it Theory of Bergman spaces}, Graduate Texts in Mathematics, 199. Springer-Verlag, New York, 2000. x+286 pp.


\bibitem{J2003}
Z. J. Jab{\l}o\'nski, Hyperexpansive composition operators, {\it Math. Proc. Camb. Phil. Soc.} {\bf 135} (2003), 513-526. \doi{10.1017/S0305004103006844}

\bibitem{JJS2012}
Z. J. Jab{\l}o\'nski, I. B. Jung, J. Stochel, Weighted shifts on directed trees, {\it Mem. Amer. Math. Soc.} {\bf 216} (2012), no. 1017, viii+106. 
\doi{10.1090/S0065-9266-2011-00644-1}

\bibitem{K2011}
P. Kumar, {\it A study of composition operators on $\ell^p$ spaces}, Thesis (PhD), Banaras Hindu University, Varanasi, 2011.


%\bibitem{L1959} 
%P. D. Lax, Translation invariant spaces, {\it Acta Math.} {\bf 101} (1959), 163-178. \doi{10.1007/BF02559553}

\bibitem{LGR2021} S. Luo, C. Gu, S. Richter, Higher order local Dirichlet integrals and de Branges–Rovnyak spaces, {\it Adv. Math.} {\bf 385} (2021), Paper No. 107748, 47 pp.
\doi{10.1016/j.aim.2021.107748}

%\bibitem{MMS2016} W. Majdak, M. Mbekhta, L. Suciu, Operators intertwining with isometries and Brownian parts of
%2-isometries, {\it Linear Algebra Appl.} {\bf 509} (2016) 168-190. \doi{10.1016/j.laa.2016.07.014}

%\bibitem{O2005}  
%A. Olofsson,  Wandering subspace theorems, {\it Integr. Equ. Oper. Theory} {\bf 51} (2005), 395-409. \doi{10.1007/s00020-003-1322-0}

\bibitem{R1988}
S. Richter, Invariant subspaces of the Dirichlet shift, {\it J. Reine Angew. Math.} {\bf 386} (1988), 205-220. 
\doi{10.1515/crll.1988.386.205}

%\bibitem{R2019} E. Rydhe, Cyclic $m$‐isometries and Dirichlet type spaces, {\it J. London Math. Soc.} {\bf 99} (2019), 733-756. \doi{10.1112/jlms.12199}

%\bibitem{Se} D. Seco, 
%A $z^k$-invariant subspace without the wandering property,  {\it J. Math. Anal. Appl.} {\bf 472} (2019), 1377-1400. \url{DOI 10.1016/j.jmaa.2018.11.081}

\bibitem{S2001}
S. Shimorin, Wold-type decompositions and wandering subspaces for operators
close to isometries, {\it J. Reine Angew. Math.} {\bf 531} (2001), 147-189. 
\doi{10.1515/crll.2001.013}

%\bibitem{S2002} S. Shimorin, Complete Nevanlinna-Pick property of Dirichlet-type spaces, {\it J. Funct. Anal.} {\bf 191} (2002), 276-296. \doi{10.1006/jfan.2001.3871}

\bibitem{SF} 
B. Sz.-Nagy, C. Foias, H. Bercovici and L. K\'erchy,  {\it Harmonic analysis of operators on Hilbert space}, Second edition. Revised and enlarged edition. Universitext. Springer, New York, 2010. xiv+474 pp. \doi{10.1007/978-1-4419-6094-8}
\end{thebibliography}
\end{document}